\documentclass{amsart}
 
  \usepackage{amsmath}

\newtheorem{thm}{Theorem}[section]

\newtheorem{lem}[thm]{Lemma}

\newtheorem{ex}{Example}[section]
\newtheorem{defi} {Definition}

\title { Control Problems and Invariant Subspaces for the Sabra Shell Model of Turbulance}

\author
{Tania Biswas and Sheetal Dharmatti}

\begin{document}
\maketitle

\centerline{\scshape Tania Biswas} 
\medskip
{\footnotesize
 \centerline{School of Mathematics, IISER Thiruvananthapuram}
   \centerline{Computer Science Building, CET Campus}
   \centerline{ Trivandrum, Kerala, India 695016}
 \centerline{ tania9114@iisertvm.ac.in}
} 

\medskip

\centerline{\scshape Sheetal Dharmatti } 
\medskip
{\footnotesize
 \centerline{School of Mathematics,  IISER Thiruvananthapuram}
   \centerline{Computer Science Building, CET Campus}
   \centerline{Trivandrum, Kerala, India 695016}
 \centerline{ sheetal@iisertvm.ac.in}
}


\begin{abstract}
Shell models of turbulence are  representation of turbulence equations in Fourier domain.  Various shell models and their existence theory along with numerical simulations have been studied earlier. In this work we study control problems related to sabra shell model of turbulence. We associate two cost functionals: one ensures minimizing turbulence in the system  and the other addresses the need of taking the flow near a priori known state. We derive optimal controls in terms of the solution of adjoint equations for corresponding linearized problems. In this work, we  also establish feedback controllers which would preserve prescribed physical constraints.  Since fluid equations have certain fundamental invariants, we would like to preserve these quantities via a control in the feedback form. We utilize the theory of nonlinear semi groups and represent the feedback control as  a multi-valued feedback term which lies in the normal cone of the convex constraint space under consideration.
\end{abstract}
\maketitle
\section{Introduction}
Mathematical modeling of turbulence is very complicated. Various theories and models are proposed in \cite{bohr}, \cite{donough}, \cite{frisch}, \cite{kandanoff}. Turbulent flows show large interactions at local levels/nodes. Hence it is suitable to model them in frequency domain or commonly known as Fourier domain. Shell models of turbulence are simplified caricatures of equations of fluid mechanics in wave-vector representation. They exhibit anomalous scaling and local non-linear interactions in wave number space. We would like to study control problem related to one such widely accepted shell model of turbulence  known as sabra shell model. 

Shell models are well known as they retain certain features of Navier Stokes Equations. The spectral form of Navier Stokes Equations motivated people to study shell models. But, unlike spectral model of Navier Stokes Equations, shell models contain local interaction between the modes, that is interaction in the short range which is important in turbulent phenomena. Several shell models have been proposed in literature.  The form of the governing equations is derived by the necessity that the helicity and energy are conserved as in the case of Navier Stokes Equations.  The most popular and well studied shell model was proposed by Gledzer and was investigated numerically by Yamada and Okhitani, which is referred as the Gledzer – Okhitani – Yamada or GOY model in short \cite{gledzer}, \cite{yamada}. The numerical experiments performed by them showed that the model exhibits an enstrophy cascade and chaotic dynamics. This garnered lot of interest in the study of shell models and many papers investigating shell models have been published since then. For more details about the shell models, we refer to \cite{Peter}.  

In this work we consider a model known as sabra shell model, introduced in \cite{L'vov et al}. To derive the form of sabra shell model, the usual construction of local interactions in $k$-space, inviscid conservation of energy and fulfillment of Liouville's theorem are used apart from the demand that the momenta involved in the triad interactions must add up to zero. The main difference of this model with respect to the GOY model lies in the number of complex conjugation operators used in the nonlinear terms. As shown in \cite{L'vov et al}, this slight change, is responsible for a difference in the phase symmetries of the two models. The sabra shell model exhibits shorter ranged correlations than the GOY model.
Apart from this difference all calculations for GOY model remains similar to the calculations for sabra shell model. Moreover, both models also share the same quadratic invariants.  Thus the results obtained in this work are equally applicable to GOY model.

  To describe the sabra shell model, the spectral spaces are divided into concentric spheres of exponentially growing radius,
\begin{equation}\label{e2}
k_n=k_0\lambda ^n
\end{equation}
with $\lambda>1$ and $k_0>0$. The one dimensional wave numbers are  denoted by $k_n$'s. The set of wave numbers contained in the $n$th sphere is called $n$th shell and $\lambda$ is the shell spacing parameter. The spectral velocity  $u_n$ is a kind of mean velocity, of the complex Fourier co-efficients of the velocity in the $n$th shell.

 The equations of motion of the sabra shell model of turbulence are given in the following form
\begin{equation}\label{e1}
\frac{du_{n}}{dt}=i(ak_{n+1}u_{n+2}
  u_{n+1}^*+bk_nu_{n+1}u_{n-1}^*-ck_{n-1}u_{n-1}u_{n-2})-\nu k_n^2u_n+f_n
\end{equation}
for $n=1,2,3...$. The boundary conditions are $u_{-1}=u_0=0$. The kinematic viscosity is represented by $\nu>0$ and $f_n$'s are the Fourier components of the forcing term. The nonlinear term defines the nonlinear interaction between the nearest nodes. The constants $a,b,c$ are chosen such that the energy and enstrophy is conserved which gives the relation $a+b+c=0$. In \cite{PBT}, Constantin, Levant and Titi have studied this model analytically and have proved the global  regularity of solutions. They rewrite an abstract formulation  of the model in $ l^2$ space and have obtained existence  and uniqueness of the strong and weak solutions for the equations in appropriate spaces.

In  \cite{PBT 1} the same authors have further studied the global existence of weak solutions of the inviscid sabra shell model and have shown that these solutions are unique for some short interval of time. Moreover, they  give a Beal-Kato-Majda type criterion for the blowup of solutions of the inviscid sabra shell model and show the global regularity of the solutions in the “two-dimensional” parameters regime.
 
Control problems associated with turbulence equations in general and shell models in particular, have not been studied widely. To our knowledge, there are no known results for the control problems associated with shell models of turbulence. In the current work our aim is to study optimal control problems associated with the sabra shell model. For, we assume the control to be acting as a forcing term in the equation and try to minimize the associated cost functional. 

 Various optimal control problems are  studied in literature for Navier Stokes Equations  and other nonlinear fluid flow problems see \cite{Tem}, \cite{sri book}, \cite{raymond} and references therein.  Here, we  study  control problems for shell model of turbulence similar to the one studied for Navier Stokes Equations.   We  consider two control problems with two different cost functionals: in the first one  aim is to minimize the vorticity in the flow which is equivalent to minimizing the turbulence. Whereas the other one aims to find a control which can steer the flow close to the a priori known desired state.

 Another important problem studied in this work is about preserving prescribed physical constraints of the model via feedback controllers. The flow equations have some invariant quantities, which are preserved, eg: energy, helicity  and enstrophy. In engineering applications, the requirement can be of ensuring the enstrophy in a specific spatial region to be kept within a bound. We wish to find a control in a feedback form which can fulfill such a requirement. For Navier Stokes Equations, flow preserving feedback controllers are studied by Barbu and Sritharan in \cite{BS}. They have shown that the controllers in the feedback form lie in the appropriate normal cone to the state space. Since the  sabra  shell model has flow preserving quantities  like helicity and enstrophy, we wish to investigate similar questions for this model. The question of finding feedback control can be looked upon as an  optimization problem.  For operator equations in abstract Hilbert spaces, the  problem of finding best  interpolation from a convex subset  has been studied and shown to be in the normal cone of the convex subset. Various other optimization problems studied in \cite{hil}, \cite{hil 1}, \cite{hil 2}, \cite{hil 3}, \cite{hil 4} give similar results. Thus even for our problem, it is suitable to look for feedback controller in the normal cone of the convex constraint space we are working with.

The paper is organized as follows: We discuss the functional setting of the problem, important properties of the operators involved and the existence result in the next section. This section has resemblance with \cite{PBT} and \cite{PBT 1}, as we are reiterating the results of these two papers. Section 3 is devoted to the study of two control problems where, we characterize the optimal control using the adjoint equation. In section 4 we prove three important theorems about flow preserving feedback controllers. At the end of this section, we demonstrate, how these theorems can be used to determine the feedback  controls, with the help of examples. We conclude the paper by summarizing our results  and list few interesting problems which can be studied further.



\section{Functional setting}
To simplify notations, we look upon $\{u_n\}$ as an element of $H=l^2(\mathbb{C} )$ and rewrite the equation \eqref{e1} in the following functional form by appropriately defining operators $A$ and $B$, 
\begin{equation} \label{e3}
\frac{du}{dt}+\nu Au+B(u,u)=f \ u(0)=u^0.
\end{equation}
For defining operators $A$ and $B$ we introduce certain functional spaces below.
For every $ u, v \; \in {H}$ the scalar product $(\cdot,\cdot)$ and the corresponding norm $|\cdot|$ are defined as,
\begin{equation*}
(u,v)=\sum\limits_{n=1}^{\infty} u_nv_n^*,\ |u|=(\sum\limits_{n=1}^{\infty} |u_n|^2)^{\frac{1}{2}}.
\end{equation*}
Let $(\phi_j)_{j=1}^{\infty}$ be the standard canonical orthonormal basis of $H$. The linear operator $A:D(A)\rightarrow H$ is defined through its action on the elements of the canonical basis of $H$ as
\begin{equation*}
A\phi_j=k_j^2\phi_j
\end{equation*}
where the eigenvalues $k_j^2$ satisfy relation \eqref{e2}. The domain of $A$ contains all those elements of $H$ for which  $| Au| $ is finite. It is denoted by  $ D(A) $ and is a dense subset of $H$. Moreover, it is a Hilbert space when equipped with graph norm 
\begin{equation}
\|u\|_{D(A)}= |Au| \ \ \forall u\in D(A). \nonumber
\end{equation}

The bilinear operator $B(u,v)$ will be defined in the following way. Let $u,v\in H$ be of the form  $u=\sum_{n=1}^{\infty}u_n\phi_n$ and  $v=\sum_{n=1}^{\infty}v_n\phi_n$.
 Then,
\begin{equation*}
B(u,v)=-i\sum_{n=1}^{\infty}(ak_{n+1}v_{n+2}
 u_{n+1}^*+bk_nv_{n+1}u_{n-1}^*+ak_{n-1}u_{n-1}v_{n-2}+bk_{n-1}v_{n-1}u_{n-2})\phi_n.
\end{equation*}
With the assumption $u^0=u_{-1}=v_0=v_{-1}=0$ and together with the energy conservation condition $a+b+c=0$, we can simplify and rewrite $B(u,v)$ as
$$B(u,u)  = -i\sum_{n=1}^{\infty}(ak_{n+1}u_{n+2}
 u_{n+1}^*+bk_nu_{n+1}u_{n-1}^*-ck_{n-1}u_{n-1}u_{n-2})\phi_n$$

With above definitions of $A$ and $B$, \eqref{e1} can be written in the form $$\frac{du}{dt} +\nu  A u + B(u,u) = f ; \; u_0 = u^0.$$  We now give some properties of $A$ and $B$.

Clearly, $A$ is positive definite, diagonal operator. Since $A$ is a positive definite operator, the powers of $A$ can be defined for every $s\in \mathbb{R}.$ 
For $ u=(u_1,u_2,...) \in H, $  define $ A^su=(k_1^{2s}u_1,k_2^{2s}u_2,...)$.\\
Furthermore we define the spaces  
$$V_s:=D(A^\frac{s}{2})={\{u=(u_1,u_2,...) \ \vert \ \sum\limits_{j=1}^{\infty} k_j^{2s}|u_j|^2<{\infty}\}}$$\\
which are Hilbert spaces equipped with the following scalar product and norm,\\
$$(u,v)_s=(A^{s/2}u,A^{s/2}v) \ \forall u,v \in V_s, \; \|u\|=(u,u)_s \ \ \forall u \in V_s.$$ 
Using above definition of the norm we can show that  $V_s \subset V_0=H \subset V_{-s}  \ \ \forall s>0$. Moreover, it can be shown that the dual space of $ V_s $ is given by $ V_{-s} $.
Domain of $A^{1/2} $ is denoted by  $V$ and is  equipped with scalar product $((u,v))=  (A^{1/2}u,A^{1/2}v) \ \ \forall u,v \in D(A^{1/2})$.  Thus we get the  inclusion 
 $$V \subset H=H' \subset V,' $$
where $ V^\prime$, the dual space of $V$ which is  identified with  $D(A^{-1/2})$.  We denote by $ \langle\cdot,\cdot\rangle$ the action of the functionals from $ V' $ on the elements of $ V $. 
Hence for every  $ u\in V $, the $H$ scalar product of $f \in H$ and $ u\in V $ is  same as the action of $f$ on $u$ as a functional in $V'$.
$$_{V'}\langle f,u\rangle_V = (f,u)_H \ \ \forall f \in H, \ \forall u \in V.$$ 
So for every $u\in D(A)$ and for every $ v \in V$, we have $((u,v))=(Au,v)= \langle u,v \rangle$ .\\
Since $D(A)$ is dense in $ V$ we can extend the definition of the operator $A:V \longrightarrow V' $ in such a way that $\langle Au,v \rangle = ((u,v)) \ \ \forall u,v \in V.$\\
In particular it follows that 
$$\|Au\|_{V'} = \|u\|_V \ \  \forall u \in V.$$

\begin{defi} \label{sfs}
Let us introduce the Sobolev functional spaces for $ m \in \mathbb{R} $
\begin{eqnarray*}
w^{m,p}&:=&\{u=(u_1,u_2,...)\  \vert \ \|A^{m/2}u\|_p =(\sum_{n=1}^{\infty}(k_n^m|u_n|)^p)^\frac{1}{p}<\infty\},\;1\leq p<\infty \nonumber \\ 
w^{m,\infty}&:=&\{u=(u_1,u_2,u_3...) \  \vert \ \|A^{m/2}u\|_\infty=(\sup(k_n^m|u_n|)<\infty)\},\; p=\infty 
\end{eqnarray*}
\end{defi}
For $u\in w^{m,p}$, the norm is defined as $$\|u\|_{w^{m,p}}=\|A^{m/2}u\|_p,$$ 
where  $\|\cdot\|_p$  is the usual norm in the  $l^p$  sequence space. $w^{m,2}$ is a Hilbert space with respect to the norm defined above. Clearly, $w^{1,2}=V$.\\
Also the inclusion $V\subset w^{1,\infty}$ is continuous (but not compact) because
\begin{align}
\|u\|_{w^{1,\infty}} &= \|A^{1/2}u\|_{\infty}  \nonumber 
\\&\leq \|A^{1/2}u\|_{2}  \nonumber 
\\&= \|u\|_V \nonumber
\end{align}

\begin{thm} \textbf{(Properties of bilinear operator $B$ )} \label{prop b}
\begin{enumerate}
\item $B:H\times V\longrightarrow H$ and $B:V\times H\longrightarrow H$ are bounded, bilinear operators. Specifically 
\begin{enumerate}
\item $|B(u,v)|\leq C_1|u|\|v\| \ \ \forall u\in H, v\in V$
\item  $ |B(u,v)|\leq C_2|v|\|u\| \ \ \forall u\in V, v\in H$
\end{enumerate} 
where \\
 $C_1=(|a|(\lambda^{-1}+\lambda)+|b|(\lambda^{-1}+1))$ \\ 
 $C_2=(2|a|+2\lambda|b|)$. 
 \item $B:H\times H\longrightarrow V'$ is a bounded bilinear operator and $\|B(u,v)\|_{V'}\leq C_1|u||v|$   $\forall u,v\in H.$
 \item $B:H\times D(A)\longrightarrow V$ is a bounded bilinear operator and for every $u\in H$ and $v\in D(A)$ \\
  $\|B(u,v)\|\leq C_3|u||Av|  \\
  C_3=(|a|(\lambda^{3}+\lambda^{-3})+|b|(\lambda+\lambda^{-2})).$ 
 \item For every $u\in H$ and $v\in V$, $Re(B(u,v),v)=0.$
 \item Let $u,v,w \in V.$ Denote $b(u,v,w)= \langle B(u,v),w \rangle$. Then 
 \begin{enumerate} 
 \item $b(u,v,w) = - b(v,u,w) $
 \item $b(v,u,w) = - b(v,w,u) . $
\end{enumerate}
\item Denote $B(u)=B(u,u)$. The map $u \longrightarrow B(u)$ is differentiable from $V$ to $V'$, and
$$ B'(u)v= B(u,v) + B(v,u) \ \ \  \forall v \in V,$$ 
$$ \ \ \ \langle B'(u)v,w\rangle_{(V',V)} = b(u,v,w) + b(v,u,w) \ \ \ \forall v,w \in V.$$
\item Let $B'(u)^*$ denote the adjoint of $B'(u)$, i.e. $\langle B'(u)v,w\rangle = \langle v,B'(u)^*w\rangle$. Then we have
$$ B'(u)^*w = -B(u,w) - B(w,u) \ \ \ \forall w\in V.$$
\end{enumerate}
\end{thm}

\begin{proof}
The proofs 1.-4. are given in \cite{PBT}. We will prove the properties 5.,6.,7.
\begin{align*}
&b(u,v,w) =-i \sum_{n=1}^{\infty} (ak_{n+1}v_{n+2}u_{n+1}^*+bk_nv_{n+1}u_{n-1}^*+ak_{n-1}u_{n-1}v_{n-2} \\
&+bk_{n-1}v_{n-1}u_{n-2})w_n^* \\
&= -i \sum_{n=1}^{\infty} ak_{n+1}v_{n+2}u_{n+1}^*w_n^*+ \sum_{n=1}^{\infty}bk_nv_{n+1}u_{n-1}^*w_n^*+ \sum_{n=1}^{\infty}ak_{n-1}u_{n-1}v_{n-2}w_n^*  \\
&+ \sum_{n=1}^{\infty}bk_{n-1}v_{n-1}u_{n-2}w_n^* \\
&= -i \sum_{n=3}^{\infty} ak_{n-1}v_{n}u_{n-1}^*w_{n-2}^*+ \sum_{n=2}^{\infty}bk_{n-1}v_{n}u_{n-2}^*w_{n-1}^*+ \sum_{n=-1}^{\infty}ak_{n+1}u_{n+1}v_{n}w_{n+2}^* \\
&+ \sum_{n=0}^{\infty}bk_{n}v_{n}u_{n-1}w_{n+1}^* \\ 
&\mbox{ Using } u_{-1}=u_0=v_{-1}=v_0=0, \mbox{ we get, } \\
&= -i \sum_{n=1}^{\infty} ak_{n-1}v_{n}u_{n-1}^*w_{n-2}^*+ \sum_{n=1}^{\infty}bk_{n-1}v_{n}u_{n-2}^*w_{n-1}^*+ \sum_{n=1}^{\infty}ak_{n+1}u_{n+1}v_{n}w_{n+2}^* \\
&+ \sum_{n=1}^{\infty}bk_{n}v_{n}u_{n-1}w_{n+1}^* 
\end{align*}
Hence,
\begin{equation*}
B^{*}(u,v) = i\sum_{n=1}^{\infty}(ak_{n+1}v_{n+2}^*u_{n+1}+bk_nv_{n+1}^*u_{n-1}+ak_{n-1}u_{n-1}^*v_{n-2}^*+bk_{n-1}v_{n-1}^*u_{n-2}^*)\phi_n^* 
\end{equation*}
\begin{eqnarray*}
\langle B^*(u,w),v^* \rangle = i\sum_{n=1}^{\infty}(ak_{n+1}w_{n+2}^*u_{n+1}v_n+bk_nw_{n+1}^*u_{n-1}v_n+ak_{n-1}u_{n-1}^*w_{n-2}^*v_n \\
+bk_{n-1}w_{n-1}^*u_{n-2}^*v_n) 
\end{eqnarray*}
Therefore, 
$$\langle B^*(u,w),v^* \rangle = -b(u,v,w)$$
Now using it repeatedly we get, 
\begin{align*}
\langle B(u,v),w\rangle &= -\langle B^*(u,w),v^* \rangle \\
&= -\langle B(u,w),v \rangle ^* \\
&= -(\langle v, B(u,w) \rangle ^*)^* \\
&= -\langle v, B(u,w) \rangle ,
\end{align*}
which proves 5(a). Similarly we can prove 5(b). \\

To prove $u \rightarrow B(u)$ is differentiable, it is enough to show that: 
$$  \sup_{w\in V , \ w\neq 0} \left( \frac{|(B(u)-B(v)-B'(u)(v-u),w)|}{\|v-u\| \|w\|} \right) \rightarrow 0 \mbox{     as     } \|v-u\| \rightarrow 0$$ with
$B'(u)(v-u) = B(v-u,u) + B(u,v-u)$. \\
For all $u,v,w \in V$ we get,
\begin{align}\label{e30}
B(u)-B(v)-B'(u)(v-u) &= B(u,u)- B(v,v)- B(v-u,u)- B(u,v-u) \nonumber \\
&= B(-v,u)- B(v,v)- B(u,v-u) \nonumber \\
&= -B(v,u-v)- B(u,v-u) \nonumber \\
&= B(-v,u-v)- B(u,u-v) \nonumber \\
&= B(u-v,u-v).
\end{align}
Now using \eqref{e30} and 2. of Theorem 2.1 we can estimate
\begin{align*}
\sup_{w\in V , \ w\neq 0} \left( \frac{|(B(u)-B(v)-B'(u)(v-u),w)}{\|v-u\| \|w\|} \right) &\leq C \frac{\|u-v\|\|u-v\|}{\|w\|\|v-u\|} \\
&\leq C \frac{\|u-v\|}{\|w\|} \rightarrow 0 \ \mbox{ as } \ \|v-u\| \rightarrow 0.
\end{align*}
This proves [6]. \\

Using 5. of Theorem 2.1 we get,   
\begin{align*}
\langle B'(u)v,w\rangle &= \langle B(u,v),w\rangle + \langle B(v,u), w\rangle \\
&= -\langle v, B(u,w) \rangle - \langle v, B(w,u) \rangle \\
&= -\langle v, B(u,w)+B(w,u) \rangle
\end{align*}
So, we can denote $B'(u)^*w = -B(u,w)-B(w,u)$. This proves 7.
\end{proof}


The existence and uniqueness for shell model of turbulence \eqref{e3} are thoroughly studied in \cite{PBT}. The proof uses mainly Galerkin approximation and Aubin's Compactness lemma. The existence of weak and strong solution for the problem as studied in \cite{PBT}[Theorem 2, Theorem 4] are stated below.
\begin{thm}\label{weak}
Let $f \in L^2([0,T],V')$ and $u^ 0\in H$. Then there exists a unique weak solution $u \in L^{\infty}([0,T],H) \cap L^2([0,T],V)$ to \eqref{e3}. Moreover the weak solution $u \in C([0,T],H)$.  
\end{thm}
\begin{thm}\label{strong}
Let $f \in L^{\infty}([0,T],H)$ and $u^0 \in V$. Then there exists a unique strong solution $u \in C([0,T],V) \cap L^2([0,T],D(A))$ to \eqref{e3}. 
\end{thm}


\section{Optimal Control Problem}
In this section, we study optimal control problem for the shell model of turbulence. We will consider two control problems with two different cost functionals. \\
Consider the equation 
\begin{equation}\label{e4}
\frac{du}{dt}+\nu Au+B(u,u)=f+g , \ u(0)=u^0,
\end{equation}
where $g$ serves as a control parameter. We choose $g\in L^2([0,T],H)$. So by Theorem  \ref{weak} if $u^0 \in H$ and $f \in L^2([0,T],V')$, then the unique solution of \eqref{e4} exists and belongs to $L^2([0,T],V)$. \\
For shell model, the vorticity of flow is the velocity derivative as mentioned in \cite{PBT 1}. We choose velocity derivative in the cost functional and minimize it so as to reduce the turbulence effect of the flow. The curl of $u$ gives the vorticity in the flow. The smaller this quantity is, the less agitated the fluid will become. So we choose the following cost functional 
\begin{equation*}
J_1(u,g) =\frac{1}{2} \int_0^T |g(t)|^2 dt + \frac{1}{2} \int_0^T |A^{1/2} u(t)|^2 dt .
\end{equation*}
Since $g \in L^2([0,T],H)$, $g(t) \in H \ \ \forall t\in[0,T]$ i.e. $|g(t)| = \sum |g_n|^2$ is finite,  so the first term in the above cost functional is well defined. $u$ denotes solution of \eqref{e4} with control $g$. Since $u \in L^2([0,T],V)$ so the second term in cost functional is also well defined.

We define the optimal control problem in the following way:
\begin{eqnarray}\label{31}
\inf \{J_1(u,g)\mid (u,g) \in L^2([0,T],V),L^2([0,T],H) , \mbox{ such that } (u,g) \nonumber \\ 
\mbox{ is the solution of } \eqref{e4}\}
\end{eqnarray}
We state and prove the existence of an optimal pair $(\bar{u},\bar{g})$ of the above control problem.

\begin{thm}\label{exis 1}
Let $u_0 \in H$ be given, then there exists at least one $\bar{g} \in L^2([0,T],H)$ and $\bar{u} \in L^2([0,T],V)\cap C([0,T],H)$ such that the functional $J_1(u,g)$ attains its minimum at $(\bar{u},\bar{g})$, where $\bar{u}$ is the solution of \eqref{e4} with control $\bar{g}.$
\end{thm}

\begin{proof}
If $g=0$, then by Theorem \ref{weak} corresponding solution $u$ exists. Therefore $\inf J_1(u,g)$ is well defined. \\
Since $0\leq \inf J_1(u,g) < \infty$, there exists a minimizing sequence $\{g_m\} \in L^2([0,T],H)$  such that  $ J_1(u_m,g_m)\rightarrow p $, where $u_m$ is solution to \eqref{e4} with control $g_m$. \\
w.l.o.g, we assume that $J_1(u_m,g_m) \leq J_1(u,0) $. This implies that
$$\frac{1}{2} \int_0^T |g_m|^2 dt + \frac{1}{2} \int_0^T |A^{1/2} u_m(x,t)|^2 dt \leq \frac{1}{2} \int_0^T  |A^{1/2} u(x,t)|^2 dt.$$
Since $ u(x,t) \in L^2([0,T],V)$, we get,
$$\frac{1}{2} \int_0^T |g_m|^2 dt \leq \frac{1}{2} \int_0^T  |A^{1/2} u(x,t)|^2 dt < \infty .$$
Similarly we get,
$$\frac{1}{2} \int_0^T |A^{1/2} u_m(x,t)|^2 dt \leq \frac{1}{2} \int_0^T  |A^{1/2} u(x,t)|^2 dt < \infty .$$
Therefore $\{g_m\}$ is bounded in $L^2([0,T],H)$ and $\{u_m\} $ is bounded in $ L^2([0,T],V).$\\
So we can find a subsequence indexed by let's say $\{g_m\}$ and $\{u_m\}$ such that $g_m \rightharpoonup  \bar{g}$ in $ L^2([0,T],H)$ and $u_m \rightharpoonup \bar{u}$ in $ L^2([0,T],V)$. \\
Therefore,
$$\frac{1}{2} \int_0^T |\bar{g}|^2 dt \leq \liminf \limits _{m\rightarrow \infty} \frac{1}{2} \int_0^T |g_m|^2 dt $$
and $$\frac{1}{2} \int_0^T |\bar{u}|^2 dt \leq \liminf \limits _{m\rightarrow \infty} \frac{1}{2} \int_0^T |u_m|^2 dt . $$
Adding the above two inequalities we get
$$J_1(\bar{u},\bar{g}) \leq \liminf \limits _{m\rightarrow \infty} J_1(u_m,g_m).$$
Hence $J_1$ attains infimum at $(\bar{u},\bar{g})$.

Now it remains to show that $\bar{u}$ is the solution of \eqref{e4} with r.h.s. $\bar{g}$, which will complete the proof. We will prove it by showing that $(\bar{u},\bar{g})$ satisfies the weak formulation of \eqref{e4} together with the initial condition i.e.
\begin{equation}\label{32}
\forall v \in V, \ \ (\frac{d\bar{u}}{dt},v) + \nu (A\bar{u},v) + b(\bar{u},\bar{u},v)  = (f,v) + (\bar{g},v), \ \ (\bar{u}(0),v)=(\bar{u}_0,v).
\end{equation}
Since $(u_n,g_n)$ is the solution of \eqref{e4}, it satisfies the weak formulation i.e. we show
\begin{equation*}
\forall v \in V, \ \   (\frac{du_m}{dt},v) + \nu (Au_m,v) + b(u_m,u_m,v)  = (f,v) + (g_m,v), \ \ (u_n(0),v)=(u_{0},v). 
\end{equation*}
Since $u_m \rightharpoonup \bar{u}$ in $ L^2([0,T],V)$, so passing to the limit we get \\
$(\frac{du_m}{dt},v) \rightarrow (\frac{d\bar{u}}{dt},v)$ and $(u_m(0),v) \rightarrow (\bar{u}(0),v).$
As $A$ is a linear operator $\nu (Au_m,v) \rightarrow  \nu (A\bar{u},v).$ \\
Similarly since $g_m \rightharpoonup  \bar{g}$ in $ L^2([0,T],H)$, so $(g_m,v)  \rightarrow (\bar{g},v).$ \\ 
We are left to evaluate the nonlinear part. It is enough to check that $|(B(u_m,u_m),v) - B(\bar{u},\bar{u}),v)| \rightarrow 0$ as $m \rightarrow \infty.$
\begin{align*}
|b(u_m,u_m,v) - b(\bar{u},\bar{u},v)| &= |b(u_m - \bar{u},u_m,v) + b(\bar{u},u_m,v) - b(\bar{u},\bar{u},v)| \\
&= |b(u_m - \bar{u},u_m,v) + b(\bar{u},u_m- \bar{u},v)| \\
&\leq  |b(u_m - \bar{u},u_m,v)| + |b(\bar{u},u_m- \bar{u},v)|
\end{align*}
Using 2 of Theorem 2.1 we get, 
$$|b(u_m,u_m,v) - b(\bar{u},\bar{u},v)| \leq C|u_m - \bar{u}||u_m|\|v\| + C|\bar{u}||u_m - \bar{u}|\|v\|.$$
We know $\{u_m\}$ is bounded. So as $m \rightarrow \infty$,    $|(B(u_m,u_m),v) - B(\bar{u},\bar{u}),v)| \rightarrow 0.$
Hence $(\bar{u},\bar{g})$ satisfies the weak formulation \eqref{32}. Therefore $(\bar{u},\bar{g})$ is the solution of \eqref{e4},
$$\frac{d\bar{u}}{dt}+\nu A\bar{u}+B(\bar{u},\bar{u})=f+\bar{g}, \ \ \bar{u}(0)= \bar{u}^0.$$
This completes the proof.
\end{proof}

In the previous theorem, we have proved the existence of an optimal solution of \eqref{31}. Next, we want to characterize the optimal control which will be in terms of the solution of the adjoint linearized  system of \eqref{e4}. We will find the linearized system of \eqref{e4} in the following theorem. We denote by $u_g$ the solution of \eqref{e4} with the control $g.$

\begin{thm}\label{def w}
Let $u^0 \in V$ be given. Then the mapping $g \rightarrow u_g$ is Gateaux differentiable in every direction $h \in L^2([0,T],H)$ and the derivative in the direction of $h$ is given by $ \langle \left(\frac{Du_g}{Dg}\right) , h \rangle$. 
For each $h \in L^2 ( [0, T], H) $, let  $w$ be the solution of following linearized problem
\begin{equation}\label{e5}
\frac{dw}{dt} + \nu Aw + B'(u_g)w = h , \ \ w(0)=0.
\end{equation}
Then for each $h \in L^2 ( [0, T], H),  $ we can  characterize $ \frac{Du_g}{Dg } $ by, $\langle \left(\frac{Du_g}{Dg}\right), h\rangle = \langle w , h\rangle $.
\end{thm}

\begin{proof}
Let us fix $u^0 \in V$ and $g,h \in L^2([0,T],H).$\\
We have to prove that $g \rightarrow u_g$ is gateaux differentiable  and the derivative in the direction of $h$ is $\langle \left(\frac{Du_g}{Dg}\right), h\rangle$. So it is enough to prove that 
\begin{equation}\label{e45}
\lim \limits_{|\lambda| \rightarrow 0} \left(\frac{|{u_{g+\lambda h}} - u_g - \lambda w|}{|\lambda|} \right)  =  0.  
\end{equation}
Set $z =  u_{g+\lambda h} - u_g - \lambda w$. 
$$\frac{dz}{dt} = \frac{du_{g+\lambda h}}{dt} - \frac{du_g}{dt} - \lambda  \frac{dw}{dt} $$
Using \eqref{e4} and \eqref{e5} we get, 
\begin{align*}
\frac{dz}{dt} &= -\nu Au_{g+\lambda h} - B(u_{g+\lambda h}) +f +g+ \lambda h + \nu Au_g + B(u_g) -f - g \\
&- \lambda (-\nu Aw - B'(u_g)w + h)\\ 
&= -\nu Az  - B(u_{g+\lambda h}) + B(u_g) + \lambda B'(u_g)w.
\end{align*}
So $z$ satisfies
\begin{equation}\label{e33}
\frac{dz}{dt}  +\nu Az + B(u_{g+\lambda h}) - B(u_g) - \lambda B'(u_g)w  =  0 , \ \ z(0)=0.
\end{equation} 
Define the operatore $Q:[0,T] \rightarrow V'$,  
$$Q(t) = B(u_{g+\lambda h}) - B(u_g) - B'(u_g)(u_{g+\lambda h} - u_g)  .$$ 
From \eqref{e33} we have $z$ is the solution of 
\begin{equation}\label{e7}
\frac{dz}{dt}  +\nu Az + B'(u_g)z = -Q(t), \  \  z(0)=0
\end{equation}
and
$$\| Q(t) \|_{V'} = \|B(u_{g+\lambda h}) - B(u_g) - B'(u_g)(u_{g+\lambda h} - u_g) \|_{V'}.$$
Using 2. and 6. of Theorem 2.1 we estimate
\begin{eqnarray}\label{e6}
\| Q(t) \|_{V'} \leq C|u_{g+\lambda h} - u_g|^2 \leq C\|u_{g+\lambda h} - u_g\|^2.
\end{eqnarray}
Since \eqref{e6} holds for all $t \in [0,T]$,
$$ | Q|_{L^2([0,T],V')} \leq C |u_{g+\lambda h} - u_g|_{L^2([0,T],V)}^2 .$$
From \eqref{e7} we get, 
\begin{equation}\label{e8}
| z|_{L^2([0,T],V)} \leq | Q|_{L^2([0,T],V')} \leq C |u_{g+\lambda h} - u_g|_{L^2([0,T],V)}^2.
\end{equation}
If we show that $$\|(u_{g+\lambda h} - u_g)\| \leq C|\lambda| ,$$ then proof is done. For, set $\hat{u} = u_{g+\lambda h} - u_g.$ So $\hat{u}$ is the solution of
$$ \frac{d\hat{u}}{dt} + \nu A\hat{u} + B(u_g,\hat{u})+ B(\hat{u},u_g) + B(\hat{u}) = \lambda h \ , \ \hat{u}(0)=0.$$ 
We know that, $\hat{u} \in L^2([0,T],V) $ and by Gronwall's lemma we get,
$$\| \hat{u} \| \leq |\lambda| \|h\|_{L^2([0,T],V)} \leq C|\lambda|. $$ 
Therefore by \eqref{e8} we get,
$$| z|_{L^2([0,T],V)} \leq C|\lambda|^2$$
Now dividing by $\lambda$ and sending $\lambda \rightarrow 0$, \eqref{e45} is proved. So we have shown that, the Gateaux derivative of $g \rightarrow u_g$ in the direction of $h$ denoted by $w$ satisfies the equation \eqref{e5} and also $w \in {L^2([0,T],V)} $.
\end{proof}

\begin{lem}\label{equa}
Let, $h_1 \in  L^2([0,T],H)$ and $w_{h_1}$ be the solution of \eqref{e5} i.e.
$$\frac{dw}{dt} + \nu Aw + B'(u_g)w = h_1 , \ \ w(0)=0.$$
Then for all $h_2 \in  L^2([0,T],H)$ we have
$$\int_0^T (h_2,w_{h_1})dt = \int_0^T (\tilde{w}_{h_2},h_1)dt$$
where $\tilde{w}_{h_2}$ is the solution of the adjoint linearized problem
\begin{equation}\label{e9}
-\frac{d\tilde{w}}{dt} + \nu A^*\tilde{w} + B'(u_g)^* \tilde{w} = h_2 , \ \ \tilde{w}(T)=0.
\end{equation}
\end{lem}

\begin{proof}
\begin{align}
\int_0^T (h_2,w_{h_1})dt &=  \int_0^T (-\frac{d\tilde{w}}{dt} + \nu A^*\tilde{w} + B'(u_g)^* \tilde{w} , w)dt \nonumber \\
 &= \int_0^T [(-\frac{d\tilde{w}}{dt},w) + (\nu A^*\tilde{w},w) + (B'(u_g)^* \tilde{w},w)]dt \nonumber \\
 &= \int_0^T (-\frac{d\tilde{w}}{dt},w)dt + \int_0^T (\nu A^*\tilde{w},w)dt + \int_0^T (B'(u_g)^* \tilde{w},w)dt \nonumber \\
 &= \int_0^T (\tilde{w},\frac{dw}{dt})dt +(\tilde{w}(T),w(T)) - (\tilde{w}(0),w(0)) + \nu \int_0^T (\tilde{w},Aw)dt \nonumber \\
 &+ \int_0^T ( \tilde{w} , B'(u_g) w)dt. \nonumber 
\end{align}
Since $w(0)=0$ and $\tilde{w}(T)=0$ we get
 
$$\int_0^T (h_2,w_{h_1})dt = \int_0^T [(\tilde{w} , \frac{dw}{dt} + Aw + B'(u_g)w)]dt $$
which implies
$$\int_0^T (h_2,w_{h_1})dt = \int_0^T(\tilde{w}_{h_2} ,h_1)dt .$$
\end{proof}

Now, we will characterize the optimal control of the problem \eqref{31} i.e. $\bar{g}$ in terms of adjoint of the linearized problem. We state and prove the following theorem.

\begin{thm}
Let $(\bar{u},\bar{g})$ be an optimal pair for the control problem \eqref{31}, then the following equality holds $$\bar{g} + \tilde{w}_h= 0,$$ where $h=A \bar{u}$ and $\tilde{w}_h$ is the solution of the linearized adjoint problem 
\begin{equation*}
-\frac{d\tilde{w}}{dt} + \nu A^*\tilde{w} + B'(u_g)^* \tilde{w} = A\bar{u}, \ \ \tilde{w}(T)=0.
\end{equation*}
\end{thm}

\begin{proof}
Let $(\bar{u},\bar{g})$ be the optimal pair for the control problem. \\
Let $F(g)=J(u_g,g).$ \\
Then we have 
\begin{align*}
F(\bar{g}+ \lambda g) - F(\bar{g}) &= J(u_{\bar{g}+\lambda g}, \bar{g}+ \lambda g)- J(u_{\bar{g}}, \bar{g})  \\       
&= \frac{1}{2} \int_0^T  |{\bar{g} + \lambda g}|^2 dt + \frac{1}{2} \int_0^T  | A^{1/2} u_{\bar{g}+ \lambda g}|^2 dt  \\
&- \frac{1}{2} \int_0^T | \bar{g}|^2 dt - \frac{1}{2} \int_0^T | A^{1/2} u_{\bar{g}}|^2 dt  \\
&= \frac{1}{2} \int_0^T  ({\bar{g} + \lambda g},{\bar{g} + \lambda g}) dt - \frac{1}{2} \int_0^T (\bar{g},\bar{g}) dt  \\
& + \frac{1}{2} \int_0^T  | A^{1/2} u_{\bar{g}+ \lambda g}|^2 dt - \frac{1}{2} \int_0^T | A^{1/2} u_{\bar{g}}|^2 dt, 
\end{align*}
which implies,
\begin{align*}
&= \frac{1}{2} \int_0^T (\bar{g},\bar{g}) dt + \frac{1}{2} \int_0^T (\bar{g},\lambda g) dt + \frac{1}{2} \int_0^T (\lambda g,\bar{g}) dt + \frac{1}{2} \int_0^T (\lambda g,\lambda g) dt \\
&- \frac{1}{2} \int_0^T (\bar{g},\bar{g}) dt  + \frac{1}{2} \int_0^T | A^{1/2} u_{\bar{g}+ \lambda g}|^2 dt - \frac{1}{2} \int_0^T | A^{1/2}  u_{\bar{g}}|^2 dt 
\end{align*}
By \cite{sri book}[Chapter 5, Page 99], considering $Q_m $ as $A^{1/2}$ we get, 
\begin{align*}
&= \frac{1}{2} \int_0^T  2 \lambda (g,\bar{g}) + \frac{1}{2} \int_0^T \lambda^2 (g,g)  +  \int_0^T \left( A^{1/2}(u_{\bar{g}})  ,  A^{1/2} \left( \frac{Du_{\bar{g}}}{D\bar{g}} \cdot \lambda g \right)  \right) \\
&= \int_0^T   \lambda (g,\bar{g}) + \frac{1}{2} \int_0^T \lambda^2 (g,g)  + \lambda \int_0^T  \left( A^{1/2} u_{\bar{g}} , A^{1/2} \left( \frac{Du_{\bar{g}}}{D\bar{g}} \cdot g \right) \right).
\end{align*}
Dividing by $\lambda$ and taking $\lambda \rightarrow 0$ we finally obtain,
\begin{eqnarray}
0 \leq F'(\bar{y}) \cdot g &=& \underset{\lambda \rightarrow 0}{\lim} \frac{F(\bar{g}+ \lambda g)- F(\bar{g})}{\lambda} \nonumber \\
&=& \int_0^T  (g,\bar{g}) + \int_0^T \left( A^{1/2} u_{\bar{g}} , A^{1/2} \left( \frac{Du_{\bar{g}}}{D\bar{g}} \cdot g \right) \right) \nonumber \\
\mbox{i.e.  } 0 \leq J_1'(\bar{g}) \cdot g &=& \int_0^T  (g,\bar{g}) + \int_0^T \left( A^{1/2} u_{\bar{g}} , A^{1/2} \left( \frac{Du_{\bar{g}}}{D\bar{g}} \cdot g \right) \right). \nonumber
\end{eqnarray}                                                                                                                                                                                                                                                                                                                                                                                                                                                                                                                                                                                                                                                                                                                                                                                                                                                                                                                                                                                                                                                                                                                                                                                                                                                                                                                 
Similarly, if we take the directional derivative of $J$ in the direction of $-g$ , we will get,                                                                                                                                                                                                                                                                                    
\begin{equation*}                                                                                                                                                                                                                                                                                                                                                                                                                                                                  
J_1'(\bar{g}) \cdot g \leq 0.
\end{equation*}
Hence,
\begin{eqnarray*}
J_1'(\bar{g}) \cdot g &=& 0  \\
\int_0^T  (g,\bar{g}) + \int_0^T  \left( A^{1/2} u_{\bar{g}} , A^{1/2} \left( \frac{Du_{\bar{g}}}{D\bar{g}} \cdot g \right) \right) &=& 0.  
\end{eqnarray*}
By theorem \ref{def w} we can write $\frac{Du_{\bar{g}}}{D\bar{g}} \cdot g = w$
\begin{eqnarray}\label{e10}
\int_0^T  (g,\bar{g}) + \int_0^T  \left( A^{1/2} u_{\bar{g}} , A^{1/2} w \right) &=& 0  \nonumber \\
\int_0^T  (g,\bar{g}) + \int_0^T  (A u_{\bar{g}}  ,w) &=& 0.
\end{eqnarray}
Using lemma \ref{equa} we get,
\begin{eqnarray*}
\int_0^T  (g,\bar{g}) + \int_0^T  (\tilde{w}_h , g) &=& 0 
\end{eqnarray*}
where $h= A u_{\bar{g}} = A\bar{u}$. Since the above equality is true for all $g \in L^2([0,T],H)$ we get,
\begin{eqnarray*}
(\bar{g},g) + (\tilde{w}_h,g) &=& 0. 
\end{eqnarray*}
Therefore,
\begin{eqnarray*}
\bar{g} + \tilde{w}_h  &=& 0 
\end{eqnarray*}
i.e.     $\bar{g} = -\tilde{w}_h$  where $h= A \bar{u} $. This completes the proof.
\end{proof}


Now we associate linear quadratic cost with the given control problem and prove the existence of control. For, define
\begin{equation}\label{e34}
J_2(u,g)=\frac{1}{2} \int_0^T  |u-u_d|^2 +\frac{\beta}{2}  \int_0^T |g|^2.
\end{equation}
where $u_d$ is the desired state. Since $g \in L^2([0,T],H)$ and $u,u_d \in C([0,T],H)$, the above cost functional \eqref{e34} is well defined. Then the control problem is defined by
\begin{eqnarray}\label{35}
\inf \{J_2(u,g)\mid (u,g) \in C([0,T],H),L^2([0,T],H) , \mbox{ such that } (u,g) \nonumber \\
 \mbox{ is the solution of } \eqref{e4}\}
\end{eqnarray}
As before we will show that an optimal pair exists and then we will characterize the optimal control.

\begin{thm}
Let $u^0 \in H$ be given. Then there exists at least one $\bar{g} \in L^2([0,T],H)$ and $\bar{u} \in C([0,T],H)\cap L^2([0,T],V)$ such that the functional $J_2(u,g)$ attains its minimum at $(\bar{g},\bar{u})$, where $\bar{u}$ is the solution of $\eqref{e4}$ with control $\bar{g}.$
\end{thm}

\begin{proof}
If $g=0$, then by Theorem 3.1 corresponding solution $u$ exists. So $J(u,0)$ exists. Therefore $\inf J_2(u,g)$ exists. \\
Let $\alpha =\inf J(g)$. \\
Since, $0\leq \alpha < \infty$, there exists a minimizing sequence $\{g_m\} \in L^2([0,T],H)$ such that $J_2(u_m,g_m) \rightarrow \alpha $, where $u_m$ is the solution of \eqref{e4} with control $g_m$. \\
w.l.o.g we assume that $J_2(u_m,g_m) \leq J_2(u,0)$. This implies, 
$$\frac{1}{2} \int_0^T  |u_m-u_d|^2 +\frac{\beta}{2} |g_m|^2 \leq \frac{1}{2} \int_0^T  |u-u_d|^2 dt. $$
Since $u,u_d \in C([0,T],H)$, so $(u-u_d) \in C([0,T],H)$. Which gives
$$\frac{\beta}{2} \int_0^T |g_m|^2 dt \leq \frac{1}{2} \int_0^T  |u-u_d|^2 dt < \infty .$$
Similarly we get,
$$\frac{1}{2} \int_0^T  |u_m-u_d|^2  \leq \frac{1}{2} \int_0^T  |u-u_d|^2 dt < \infty.$$
Therefore, $\{g_m\}$ is bounded in $L^2([0,T],H)$ and $\{u_m\} $ is bounded in $ C([0,T],H) \cap L^2([0,T],V)$. 
So we can find subsequences indexed by let's say $\{g_m\}$ and $\{u_m\}$ such that $g_m \rightharpoonup  \bar{g}$ in $ L^2([0,T],H)$ and $u_m \rightharpoonup \bar{u}$ in $ C([0,T],H) \cap L^2([0,T],V)$. Therefore
$$ \int_0^T |\bar{g}|^2 dt \leq \liminf \limits _{m\rightarrow \infty} \int_0^T |g_m|^2 dt $$
and $$\int_0^T |\bar{u} - u_d|_H^2 dt \leq \liminf \limits _{m\rightarrow \infty} \int_0^T |u_m - u_d|_H^2 dt . $$ 
Adding the above two inequalities we get,
$$ J_2(\bar{u},\bar{g}) \leq \liminf \limits _{m\rightarrow \infty} J_2(u_m,g_m)\ = \ \alpha.$$
Hence,   $J_2(\bar{u},\bar{g}) \ = \ \alpha.$ \\
Now we have to show that $\bar{u}$ is the solution of $\eqref{e4}$ with control $\bar{g}$. This part of the proof is exactly similar to part of the proof of the theorem \ref{exis 1} and we skip the details here. 
\end{proof}

The next theorem characterizes the optimal control in terms of solution of the adjoint equation.

\begin{thm}
Let $(\bar{u},\bar{g})$ be an optimal pair for the control problem \eqref{35}, then the optimal control $\bar{g}$ is given by $$\bar{g} = -\frac{1}{\beta} \tilde{w}_h$$ 
where $h=\bar{u} - u_d$ and $\tilde{w}_h$ is the solution of the linearised adjoint problem 
\begin{equation*}
-\frac{d\tilde{w}}{dt} + \nu A^*\tilde{w} + B'(u_g)^* \tilde{w} = \bar{u} - u_d  \ , \ \tilde{w}(T)= \bar{u}(T) - u_d(T).
\end{equation*}
\end{thm}

\begin{proof}
Let, $(\bar{u},\bar{g})$ be the optimal pair for the control problem \eqref{35}. \\
Let, $F(g)=J_2(u_g,g).$ \\
Then we have, 
\begin{align*}
F(\bar{g}+ \lambda g) - F(\bar{g}) &= J_2(u_{\bar{g}+\lambda g}, \bar{g}+ \lambda g)- J_2(u_{\bar{g}}, \bar{g})  \\       
&= \frac{1}{2} \int_0^T  |u_{\bar{g}+\lambda g}-u_d|^2 + \frac{\beta}{2} \int_0^T  |{\bar{g} + \lambda g}|^2 dt -\frac{1}{2} \int_0^T  |u_{\bar{g}} -u_d|^2 \\
&+ \frac{\beta}{2} \int_0^T  | \bar{g}|^2 dt  \\
&= \frac{1}{2}\int_0^T (u_{\bar{g}+\lambda g} - u_d,u_{\bar{g}+\lambda g} - u_d) + \frac{\beta}{2} \int_0^T (\bar{g} + \lambda g,\bar{g} + \lambda g) \\
&- \frac{1}{2}\int_0^T (u_{\bar{g}} -u_d,u_{\bar{g}} -u_d) + \frac{\beta}{2} \int_0^T (\bar{g},\bar{g})\\
&= \frac{1}{2}\int_0^T (u_{\bar{g}+\lambda g} - u_{\bar{g}},u_{\bar{g}+\lambda g} + u_{\bar{g}} -2u_d) + \frac{\beta}{2} \int_0^T  2\lambda (g,\bar{g}) \\
&+ \frac{\beta}{2} \int_0^T \lambda^2 (g,g) \\
&= \frac{1}{2}\int_0^T (u_{\bar{g}+\lambda g} - u_{\bar{g}},u_{\bar{g}+\lambda g} - u_{\bar{g}}) + \int_0^T (u_{\bar{g}+\lambda g} - u_{\bar{g}},u_{\bar{g}} -u_d) \\ 
&+ \beta \int_0^T \lambda (g,\bar{g}) + \frac{\beta}{2} \int_0^T \lambda^2 (g,g)   \\
&= \frac{1}{2}\int_0^T  |u_{\bar{g}+\lambda g} - u_{\bar{g}}|^2 + \int_0^T (u_{\bar{g}+\lambda g} - u_{\bar{g}},u_{\bar{g}} -u_d) + \frac{\beta}{2} \int_0^T  \lambda^2 (g,g) \\
&+  \beta \int_0^T  \lambda (g, \bar{g}). 
\end{align*}
We know that $u_{\bar{g}+\lambda g} - u_{\bar{g}} = \langle \frac{Du_{\bar{g}}}{D\bar{g}} , \lambda g \rangle$. By theorem \ref{def w} we can write $u_{\bar{g}+\lambda g} - u_{\bar{g}} = \lambda w_g$, which gives, 
$$F(\bar{g}+ \lambda g) - F(\bar{g}) = \frac{{\lambda}^2}{2}\int_0^T |w_g|^2 + \int_0^T \lambda (w_g,u_{\bar{g}} -u_d) + \frac{\beta}{2} \int_0^T  \lambda^2 (g,g) +  \beta \int_0^T  \lambda (g ,\bar{g})$$
Dividing by $\lambda$ and taking $\lambda \rightarrow 0$ we obtain,
\begin{eqnarray*}
0 \leq F_1'(\bar{g}) \cdot g &=& \underset{\lambda \rightarrow 0}{\lim} \frac{F(\bar{g}+ \lambda g)- F(\bar{g})}{\lambda} \\
&=& \int_0^T  (w_g,u_{\bar{g}} -u_d) +\beta \int_0^T (g,\bar{g}).  
\end{eqnarray*}                                                                                                                                                                                                                                                                                                                                                                                                                                                                                                                                                                                                                                                                                                                                                                                                                                                                                                                                                                                                                                                                                                                                                                                                                                                                                                                 
Using Lemma \ref{equa} we get,
$$0 \leq F_1'(\bar{g}) \cdot g = \int_0^T  (g, \tilde{w}_h) +\beta \int_0^T  (g,\bar{g})$$
where $h=u_{\bar{g}} -u_d$ i.e.
$$0 \leq J_1'(\bar{g}) \cdot g = \int_0^T  (g,\tilde{w}_h) +\beta \int_0^T  (g,\bar{g}).$$ 
Similarly if we take the directional derivative of $J$ in the direction of $-g$, we will get,                                                                                                                                                                                                                                                                                    
\begin{equation*}                                                                                                                                                                                                                                                                                                                                                                                                                                                                  
J_1'(\bar{g}) \cdot g \leq 0
\end{equation*}
Hence,
\begin{align*}
J_1'(\bar{g}) \cdot g &= 0  \\
\int_0^T  (g,\tilde{w}_h) +\beta \int_0^T  (g,\bar{g}) &= 0  
\end{align*}
Since the above equality is true for all $g \in L^2([0,T],H)$, we get,
$$\bar{g} = -\frac{1}{\beta} \tilde{w}_h$$
where $h=u_{\bar{g}} -u_d$. This completes the proof.
\end{proof}


\section{Invariant Subspace.} 
In this section our aim is to find feedback controllers such that certain physical properties associated with the flow are preserved. Let $K$ be a given closed convex subset of the state space. We wish to find a controller $g$ such that, whenever an initial value of flow equation lies in $K$ so does the solution of the controlled equation \eqref{e4}. In many cases the controller $g$ which preserves flow in $K$ lies in the normal cone to $K$. Under some assumptions on $K$ we show that it is possible to find flow preserving feedback controllers. Later on for specific examples of $K$ we show how these theorems can be applied. Our proofs use the theory of nonlinear partial differential equation associated to maximal monotone operator. \\
In the following theorem, we assume that $K$ is invariant under the operator $(I+\lambda K)^{-1}$. This assumption helps us to get the control which belongs to the normal cone to the set $K$, so that $K$ is invariant with respect to initial value. \\
Before stating the theorem we define necessary terms.

\begin{defi}\textbf{Normal Cone.} 
Let $K$ be a non-empty convex subset of a  Hilbert space $H$ and $x \in K$. Then the normal cone to $K$ at $x$ is defined as $N_K(x)=\{y\in H : ( y,x - z ) \geq 0 \ \forall z \in K \},$ as shown in the figure.
\end{defi}

\begin{defi}\textbf{Quasi m-accretive.}\label{QM}
 An operator $A:H \rightarrow H$ is called accretive if for every pair $[x_1 , y_1],[x_2 , y_2] \in H \times H$,  $(x_1 - x_2,y_1 - y_2) \geq 0$ holds.\\
$A$ is called m-accretive if $R(\lambda I+ A) = H  \ \ \forall \ \lambda>0$. \\
The operator $A$ is called quasi m-accretive if $(A+\mu I)$ is m-accretive for some $\mu>0$.
\end{defi}

\begin{thm}\label{thm 1}
 Let $K$ be a closed convex subset of $H$ such that $0 \in K$ and $(I+\lambda A)^{-1} K \subset K \  \forall \ \lambda > 0$. Let $u^0 \in D(A)\cap K$ and $f \in W^{1,1}([0,T],H) \cap L^{2}([0,T],H)$. Then, there exists a feedback controller $g \in L^ \infty ([0,T],H)$ and $g(t) \in N_K(u(t))$ a.e. $t \in [0,T]$ such that the corresponding solution to the closed loop system
\begin{equation}\label{e36}
\frac{du}{dt}+\nu Au+B(u,u)=f+g , \ \ u(0)=u^0
\end{equation} 
satisfies $u \in W^{1,\infty}([0,T],H) \cap L^ \infty ([0,T],K \cap D(A)) \cap C([0,T],V).$ Moreover,
 \begin{equation}\label{e11}
\frac{d^+ u(t)}{dt} + (-f(t)+ \nu Au(t) + B(u(t) + N_K(u(t)))^0 = 0 , u(0)=u^0
 \end{equation}
and $u(t) \in K \   \forall \ t\in [0,T].$ \\
Here $N_K(u(t)):= \{w\in H ; (w,u-z) \geq 0 , \forall z \in K\}$ is the normal cone to $K$ at $u(t)$. $u \rightarrow (-f(t)+ \nu Au(t) + B(u(t)) + N_K(u(t)))$ is a multivalued map. Define $(-f(t)+ \nu Au(t) + B(u(t)) + N_K(u(t)))^0$ to be the projection of origin on $(-f(t)+ \nu Au(t) + B(u(t)) + N_K(u(t)))$ which is of minimum norm. \\
So by \eqref{e11} we can deduce the feedback controller $g$ is given by,
 \begin{align}\label{e37}
 g(t)= -f(t) &+ \nu Au(t) + B(u(t))) + N_K(u(t)) \nonumber \\
 &-  (-f(t)+ \nu Au(t) + B(u(t))) + N_K(u(t)))^0  \ \ \ \forall \ t \in [0,T). 
 \end{align}
\end{thm}

\begin{proof}
Let us define the modified nonlinear map $B_N ( \cdot ):V \longrightarrow V^{'},$ 
\[B_N ( \cdot ) :=
  \begin{cases}
    B(u)      & \quad \text{if }\|u\| \leq N\\
    {\left( \frac{N}{ \|u\|} \right) }^{2}B(u)  & \quad \text{if } \|u\| > N\\
  \end{cases}
\]
First, we will prove that the operator $u \rightarrow (-f+ \nu Au +B_N(u)+N_K(u(t)))$, is quasi m-accretive in $H$. By definition \ref{QM} we have to prove that there exist $\alpha_N > 0$ such that $u \rightarrow (-f+ \nu Au +B_N(u)+N_K(u(t)) + \alpha _N u)$ is m-accretive in $H$. \\
From \cite{Barbu semi}[Chapter II, Theorem 2.1] we know that $u \rightarrow N_K(u)$ is maximal monotone. We prove $\nu A + B_N$ is m-accretive as a first step of the proof. For we define $\nu A + B_N$ as a new operator $\Gamma_N$. \\
Consider the operator $\Gamma _N: D(\Gamma_N) \rightarrow H$ defined by,
\begin{equation*}
\Gamma _N = \nu A + B_N .
\end{equation*}
We claim that $D(\Gamma_N)=D(A)$. Clearly, $ D(\Gamma_N) \subset D(A) $. So it remains to show the other inclusion. \\
For, let $x \in D(A)$. Using 3 of Theorem \ref{prop b} we get,
\begin{align*}
| \Gamma_N(x)| &= |\nu A(x) + B_N(x) | \\
&\leq |\nu A(x)| +  |B_N(x)| \\
&\leq \nu |A(x)| + C_3|x||A(x)| < \infty.
\end{align*} 
Which proves $D(\Gamma_N)=D(A)$. \\
Now, we will show that, $u \rightarrow \Gamma _N(u)$ is m-accretive. By definition \ref{QM}, it is enough to show that $u \rightarrow (\Gamma _N(u) + \alpha_N I)$ is accretive, for some $\alpha_N >0$. \\ 
For $\|u\| \leq N$ , $\|z\| \leq N$, using 2. of theorem \ref{prop b} we get,
\begin{align*}
|(B_N(u) - B_N(z) , u-z)| &= |b(u,u-z,u-z)| + |b(u-z,z,u-z)| \\
 &\leq  C|u-z||z|\|u-z\|.
\end{align*} 
Since $\|u\| \leq N$ , $\|z\| \leq N$ and using Young's inequality we can estimate,
\begin{align}\label{e12}
|(B_N(u) - B_N(z) , u-z)| &\leq C_N|u-z| \|u-z\|   \nonumber \\
&\leq C_N|u-z|^{2} + \frac{\nu}{2} \|u-z\|^{2}.
\end{align}
Similarly, for $\|u\| > N , \|z\| > N$ we get,
\begin{align}
|(B_N(u) - B_N(z) , u-z)| &= |{\left( \frac{N}{\|u\|}\right)}^{2} b(u,u,u-z) - {\left( \frac{N}{\|u\|}\right)}^{2} b(z,z,u-z)| \nonumber \\
&= |{\left( \frac{N}{\|u\|}\right)}^{2} b(u,u,u-z) - {\left( \frac{N}{\|u\|}\right)}^{2} b(z,z,u-z) \nonumber \\
&+ {\left( \frac{N}{\|u\|}\right)}^{2} b(z,z,u-z) - {\left( \frac{N}{\|z\|}\right)}^{2} b(z,z,u-z)| \nonumber \\ 
&\leq {\left( \frac{N}{\|u\|}\right)}^{2} |b(u-z,u,u-z)| + {\left( \frac{N}{\|u\|}\right)}^{2} |b(z,u-z,u-z)|  \nonumber \\
&+ |{\left( \frac{N}{\|u\|}\right)}^{2} - {\left( \frac{N}{\|z\|}\right)}^{2}| |b(z,z,u-z)|.\nonumber
\end{align}
Further using 2. , 5. of Theorem \ref{prop b} and Young's inequality we get,
\begin{align}\label{e13}
|(B_N(u) - B_N(z) , u-z)| &\leq {\left( \frac{N}{\|u\|}\right)}^{2} |b(u-z,u,u-z)|  \nonumber \\
&+ \frac{N^{2}|{\|u\|}^{2} - {\|z\|}^{2}|}{{\|u\|}^{2}{\|z\|}^{2}}|b(z,z,u-z)|  \nonumber \\
&\leq {\left( \frac{N}{\|u\|}\right)}^{2} |b(u-z,u,u-z)| + \frac{N^{2}\|u-z\|}{\|z\|{\|u\|}^{2}}|b(z,u-z,z)|  \nonumber \\
& + \frac{N^{2}\|u-z\|}{\|u\|{\|z\|}^{2}}|b(z,u-z,z)| \nonumber \\
&\leq C{\left( \frac{N}{\|u\|}\right)}^{2} |u-z||u|\|u-z\| +C\frac{N^{2}\|u-z\|}{\|u\|{\|z\|}^{2}}|z||u-z|\|z\| \nonumber \\
& +   C\frac{N^{2}\|u-z\|}{\|z\|{\|u\|}^{2}}|z||u-z|\|z\| \nonumber \\
&\leq C_N|u-z|\|u-z\| + C'_N|u-z|\|u-z\| + C''_N|u-z|\|u-z\| \nonumber \\
&\leq C_N{|u-z|}^{2} + \frac{\nu}{2}{\|u-z\|}^{2}.
\end{align}
If $\|u\| > N, \ \|z\| \leq N$ or $\|u\| \leq N, \ \|z\| > N$ then similar calculations as above yeild estimate as in \eqref{e12},\eqref{e13}. \\ 
Consider $(\Gamma_N + \lambda)$ for $\lambda >0$,
\begin{align*}
\left( (\Gamma_N + \lambda)u - (\Gamma_N + \lambda)z, u-z \right) = \left( (\nu A + B_N + \lambda)u - (\nu A + B_N + \lambda)z, u-z\right) \nonumber \\
= \nu \left(Au-Az,u-z\right) + \left(B_Nu - B_Nz,u-z\right) + \lambda \left( u-z,u-z\right). \nonumber 
\end{align*}
Since $A$  is  m-accretive, $\left( Au-Az,u-z \right) \geq \nu {\|u-z\|}^2$ and using \eqref{e12},\eqref{e13} we get,
\begin{align*}
\left( (\Gamma_N + \lambda)u - (\Gamma_N + \lambda)z, u-z \right)&\geq \nu {\|u-z\|}^{2} - C{|u-z|}^{2} - \frac{\nu}{2} {\|u-z\|}^{2} + \lambda {|u-z|}^{2} \nonumber \\ 
&= \frac{\nu}{2} {\|u-z\|}^{2} + \left( \lambda-C\right) {|u-z|}^{2} \nonumber \\
&\geq 0 \ \ \mbox{if} \ \ \lambda>C.
\end{align*}
Therefore, if we choose $\alpha_N > C > 0$, $u \rightarrow \left( \Gamma _N+ \alpha_N I\right)$ will be accretive.\\
Now it remains to show that, for fixed $\mu > \alpha_N$
\begin{equation*}
R\left( \mu u + \nu Au + B(u) + N_K(u)\right) = H .
\end{equation*}
This would be proved using Hille-Yosida theorem. For, consider the Yosida approximation of $N_K$,
\begin{equation*}
F_\lambda = \frac{1}{\lambda} \left(I-(I+\lambda N_K)^{-1}\right).
\end{equation*}
Let $f\in H$ be arbitary but fixed. Then by Minty's theorem from \cite{Barbu semi} there exists a unique $u_\lambda \in D(A)$ such that
\begin{equation}\label{e14}
 \mu u_\lambda + \nu Au_\lambda + B(u_\lambda) + F_{\lambda}(u_\lambda) = f
\end{equation}
Taking inner product of \eqref{e14} with $u_\lambda$ we get,
\begin{eqnarray}
\mu{|u_\lambda|}^{2} + \nu {\|u_\lambda \|}^{2} + \left(B_N(u_\lambda),u_\lambda \right) + \left(F_{\lambda}(u_\lambda),u_\lambda \right)  = (f,u_\lambda). \nonumber 
\end{eqnarray}
By 2. of Theorem \ref{prop b} and $\left(F_{\lambda}(u_\lambda),u_\lambda \right) \geq 0 $ we get,
$$  \mu{|u_\lambda|}^{2} + \nu {\|u_\lambda \|}^{2} \leq \frac{1}{2\mu}|f|^2 + \frac{\mu}{2}|u_\lambda|^2 .$$
This yeilds,
\begin{eqnarray}\label{e15}
 \frac{\mu}{2}{|u_\lambda|}^{2} + C_{\nu}{\|u_\lambda \|}^{2} \leq C.
\end{eqnarray}
Taking inner product of \eqref{e14} with $Au_\lambda$ we get,
$$\mu {\|u_\lambda \|}^{2} + \nu {|Au_\lambda|}^{2} + \left(B_N(u_\lambda),Au_\lambda \right) + \left(F_\lambda (u_\lambda),Au_\lambda \right) = (f,Au_\lambda). $$
Together with $(I+\lambda A)^{-1} K \subset K \  \forall \ \lambda > 0$ and \cite{Barbu semi}[Chapter IV, Proposition 1.1] we get $\left(F_{\lambda}(u_\lambda), Au_\lambda \right) \geq 0 .$ Therefore, Using 3. of Theorem \ref{prop b},
\begin{eqnarray*}
\mu {\|u_\lambda \|}^{2} + \nu {|Au_\lambda|}^{2} &\leq& |b(u_\lambda,Au_\lambda,u_\lambda)| + |f||Au_\lambda| \\
 &\leq& C|u_\lambda||Au_\lambda|\|u_\lambda\| + |f||Au_\lambda| \\
&\leq& C_{\nu}{|u_\lambda|}^{2} {\|u_\lambda\|}^{2}+ \frac{ \nu}{2}{|Au_\lambda|}^{2} + \frac{1}{\nu}{|f|}^2  
\end{eqnarray*}
Hence \eqref{e15} gives, 
\begin{eqnarray}\label{e16}
C_{\mu} {\|u_\lambda \|}^{2} + \frac{\nu}{2} {|Au_\lambda|}^{2} \leq C. 
\end{eqnarray}
Therefore, we get from \eqref{e15} and \eqref{e16},
\begin{eqnarray}\label{e17}
{\|u_\lambda \|}^{2} + {|Au_\lambda|}^{2} \leq C \ \ \forall \lambda>0.
\end{eqnarray}
Also, for fixed $N$, $B_N$ and $F_{\lambda}$ are bounded linear operators. So we get respectively, 
\begin{eqnarray}\label{e18}
|B_N(u_\lambda)| \leq C_N \mbox{ and } |F_\lambda u_\lambda | \leq C_N. 
\end{eqnarray}
Therefore by \eqref{e17} we can conclude that there exists a subsequence (let's denote by $\lambda$) such that, $u_\lambda \rightarrow u$ strongly in $V$ and $Au_\lambda \rightarrow Au$ weakly in $H$. \\
Moreover, by \eqref{e18} we conclude that $F_\lambda u_\lambda \rightarrow \gamma$ weakly in $H$ and $B_N(u_\lambda) \rightarrow B_N(u)$ weakly in $H$. \\
Since, $F_{\lambda}$ are the Yosida approximation of $N_K$, so we get $\gamma \in N_K(u)$ and $u$ is the solution of
\begin{equation*}
 \mu u + \nu Au + B_N(u) + N_K(u) = f.
\end{equation*}
Therefore, $u \rightarrow  \mu I + \nu Au + B_N(u) + N_K(u)$ is quasi m-accretive. Also $u^0 \in D(A).$ So by \cite{barbu mono}[Chapter 4, Theorem 4.5, Theorem 4.6] we get, 
\begin{equation*}
\frac{du}{dt} +\nu Au + B_N(u) + N_K(u) = f, \ \ u(0)=u^0 \ \ a.e.  \ t\in [0,T],
\end{equation*}
has a unique solution $u_N \in W^{1, \infty}([0,T],H) \cap L^{\infty}([0,T],D(A) \cap K) \cap C([0,T],V)$ which satisfies,
\begin{equation*}
\frac{d^{+}u_N}{dt} + {\left(\nu Au_N + B_N(u_N) + N_K(u_N) -f \right)}^{0} = 0 \ \  \forall t \in [0,T].
\end{equation*}
Now, we will show that $\|u_N\| < C$ for some large N. For, taking inner product of  
\begin{equation}\label{e39}
\frac{du_N}{dt} +\nu Au_N + B_N(u_N) + N_K(u_N) = f 
\end{equation}
with $u_N$ and using the fact  that $\left(N_K(u_N),u_N\right)\geq 0 $  we get,
\begin{align}\label{e38}
&\frac{d}{dt} {|u_N(t)|}^{2} + \nu {\|u_N\|}^{2} + \left(N_K(u_N),u_N\right) \leq  + |(f,u_N)| \nonumber \\
&\frac{d}{dt} {|u_N(t)|}^{2} + \nu {\|u_N\|}^{2} \leq  \frac{1}{2}{|f|}^{2} + \frac{1}{2} {|u_N|}^{2}  
\end{align}
Similarly taking inner product of \eqref{e39} with $Au_N$ and  we get,
$$\frac{d}{dt}\|{u_N(t)}\|^{2} + \nu {|Au_N|}^{2}  + \left(N_K(u_N),Au_N\right) \leq  -\left(B_N(u_N),Au_N\right) + \left(f,Au_N\right) $$
Using  $(I+\lambda A)^{-1} K \subset K \  \forall \ \lambda > 0$ and \cite{Barbu semi}[Chapter IV, Proposition 1.1] we get \\ $\left(\eta,Au_N\right) \geq 0 \ \  \forall \eta \in N_K(u_N)$. Therefore,
$$\frac{d}{dt}\|{u_N(t)}\|^{2} + \nu {|Au_N|}^{2} \leq - \left(B_N(u_N),Au_N\right) + \left(f,Au_N\right) .$$

Using Theorem \ref{prop b} we estimate $\left(B_N(u_N),Au_N\right)$ as following
\begin{align*}
-\left(B_N(u_N),Au_N\right) &\leq |\left(B_N(u_N),Au_N\right)|  \\
&\leq |B_N(u_N)| |Au_N|  \\
&\leq C|u_N| \|u_N\| |Au_N|  \\
&\leq C_{\nu}|u_N|^2 \|u_N\|^2 + \frac{\nu}{4}|Au_N|^2.
\end{align*}
Further we get,
\begin{eqnarray}\label{e40}
\frac{d}{dt}\|{u_N(t)}\|^{2} + \frac{\nu}{2} {|Au_N|}^{2} \leq C_{\nu}|u_N|^2 \|u_N\|^2 +  \frac{1}{\nu}|f|^2  .
\end{eqnarray}
Using Grownwall's lemma to \eqref{e38} we get,
\begin{equation*}
|u_N(t)|^{2}   \leq \exp{\frac{t}{2}}(|u^0|^{2}  + \frac{1}{2}\int_0^t {|f(s)|}^2 ds)  \ \ \forall  t\in [0,T].
\end{equation*}
Integrating  over $[0,t]$ and using $f\in L^{2}([0,T],H)$ we get,
\begin{equation*}
|u_N(t)|^{2} + \nu \int_0^t {\|u_N(s) \|}^{2}ds  \leq C  \ \ \forall  t\in [0,T].
\end{equation*}
Using Grownwall's lemma to \eqref{e40} we get,
\begin{eqnarray*}
\|{u_N(t)}\|^{2} \leq \left( {\|u^0\|}^{2} + \frac{1}{\nu}\int_0^t {|f(s)|}^{2} ds \right) \exp{ \left( C_{\nu} \int_0^t {|u_N(s)|}^2 ds\right)} \ \ \forall t\in [0,T]. 
\end{eqnarray*}
Integrating over $[0,t]$ and using $f\in L^{2}([0,T],H)$ and $u_N \in L^{\infty}([0,T],H)$ leads to,
\begin{eqnarray*}
\|{u_N(t)}\|^{2} + \frac{\nu}{2} \int_0^t {|Au_N(s)|}^{2}  \leq C
\end{eqnarray*}
So we have,
\begin{eqnarray*}
\|{u_N(t)}\|^{2} \leq C \ \ \forall t \in [0,T]. 
\end{eqnarray*}
So, for $N$ large enough, (such that $N > C$) ; $\|{u_N(t)}\| \leq C$ on $t\in [0,T]$. Hence, we will get $B_N(u) = B(u)$ for all $t \in (0,T)$ and $u_N = u$ is the solution to \eqref{e36}. \\
This proves the theorem.
\end{proof}

In theorem \ref{thm 1} we require that $(I+\lambda A)^{-1} K \subset K \  \forall \ \lambda > 0$. This ia a very strong assumption and may not be satisfied in practical problems. We want to relax this condition. If we do not assume that $K$ is invariant under $(I+\lambda A)^{-1}$, we still get a result but of weaker form. We assume $K \subset V$ and show that the feedback controller $g$ lies in $L^2([0,T],V')$, can be found which will ensure that trajectory does not leave $K$.

\begin{thm}\label{thm 2}
 Let, $K$ be a closed convex subset of $V$ such that $0 \in K.$ Let $u^0 \in D(A) \cap K$ and $f \in W^{1,2}([0,T],H)$. Then there exists a feedback controller $g \in L^2([0,T],H)$ and $g(t) \in -N_K^*(u(t))$ a.e. $t \in [0,T)$, such that the corresponding solution  to the closed loop system \eqref{e36} i.e.
\begin{equation*}
\frac{du}{dt}+\nu Au+B(u)=f+g , \ \ u(0)=u^0
\end{equation*}  
satisfies $u \in W^{1,\infty}([0,T],H) \cap W^{1,2}([0,T],V)$. Moreover, 
\begin{equation}\label{e21}
\frac{d^+ u(t)}{dt} + (-f(t)+ \nu Au(t) + B(u(t))) + N_K^*(u(t)))^0 = 0 , \ \ u(0)=u^0
\end{equation}
and $u(t) \in K \ \  \forall t\in [0,T]$. \\
For each $t$, $N_K^*(u(t))$ is defined by $N_K^*(u(t)) := \{w\in V' ; (w,u-z) \geq 0 , \ \forall z \in K\}$, is the $V'$ valued normal cone to $K$ at $u(t)$. Similar to the theorem \ref{thm 1} $u \rightarrow (-f+ \nu Au(t) +B(u(t))+N_K^*(u(t)))$ is a multivalued map and we define for each $u(t)$, $ (-f(t)+ \nu Au(t) + B(u(t))) + N_K^*(u(t)))^0$ as the projection of origin onto the closed convex set $(-f+ \nu Au(t) +B(u(t))+N_K^*(u(t)))$. \\
By \eqref{e21} we get the feedback controller $g$ as,
 \begin{align}
 g(t)= -f(t) &+ \nu Au(t) + B(u(t)) + N_K^*(u(t)) \nonumber \\
& - (-f(t)+ \nu Au(t) + B(u(t)) + N_K^*(u(t)))^0 \ \ \ \forall t \in [0,T).
 \end{align}
\end{thm}

\begin{proof}
We have to show that $u \rightarrow (-f+ \nu Au +B_N(u)+N_K^*(u))$ is quasi m-accretive. By \cite{Barbu semi}[Chapter II, Theorem 2.1] we know $u \rightarrow N_K^*(u)$ is maximal monotone in $V \times V^*$. Hence $u \rightarrow (-f+ \nu Au +B_N(u)+N_K^*(u) + \alpha_N u)$ is accretive for some $\alpha_N > 0$. We can prove it as in the proof of Theorem \ref{thm 1}.  Therefore this map is m-accretive by definition \ref{QM}. Also we can show that $u \rightarrow (-f+ \nu Au +B_N(u)+N_K^*(u) + \alpha_N I)$ is coercive. So by \cite{brezis}[Chapter II, Example 2.3.7] its restriction to $H$ is maximal monotone in $H \times H$. This gives, $u \rightarrow (-f+ \nu Au +B_N(u)+N_K^*(u))$ is quasi m-accretive by definition \ref{QM}. So by \cite{Barbu ns}[Chapter 1, Theorem 1.15, Theorem 1.16] we get,
\begin{equation*}
\frac{du(t)}{dt} + \nu Au(t) + B_N(u(t)) + N_K^*(u(t)) = f(t), \ \ u(0)=u^0 \ a.e. \ t \in [0,T]
\end{equation*}
has a unique solution $u_N \in W^{1,\infty}([0,T],H) \cap L^{2}([0,T],V)$. Moreover $u_N$ satisfies
\begin{equation*}
\frac{d^{+}u_N(t)}{dt} + {\left(\nu Au_N(t) + B_N(u_N(t)) + N_K^*(u_N(t)) -f(t) \right)}^{0} = 0, \ \  \forall t \in [0,T].
\end{equation*}
Taking inner product of 
\begin{equation}\label{22}
\frac{du_N(t)}{dt} + \nu Au_N(t) + B_N(u_N(t)) + N_K^*(u_N(t))=f(t) , \ \ u_N(0)=u^0.
\end{equation}
 with $u_N$ we get,
\begin{align*}
& \frac{d}{dt}|u_N(t)|^2 +\nu \|u_N(t)\|^2 + (B_N(u_N(t)),u_N(t))+ (N_K^*(u_N(t)),u_N(t)) = (f,u_N(t)) \\
& \frac{d}{dt}|u_N(t)|^2 +\nu \|u_N(t)\|^2 \leq  (f(t),u_N(t))   \\
& \frac{d}{dt}|u_N(t)|^2 + \nu \|u_N(t)\|^2 \leq  \frac{1}{2}|f(t)|^2 + \frac{1}{2} |u_N(t)|^2 . 
\end{align*}
Using Grownwall's lemma and $f \in W^{1,2}([0,T],H) $ and by integrating we get,
\begin{equation}\label{23}
|u_N(t)|^2 + \nu \int_0^t \|u_N\|^2 \leq C.
\end{equation}
Now, we will differentiate \eqref{22} with respect to $t$ to get,
\begin{equation}\label{24}
\frac{du_N'(t)}{dt} + \nu(Au_N(t))' + (B_N(u_N(t)))' + (N_K^*(u_N(t)))'= f'(t).
\end{equation}
Taking inner product of \eqref{24} with $u_N'$ we get,
$$(\frac{du_N'(t)}{dt},u_N'(t)) + \nu ((A(u_N(t)))', u_N'(t)) + ((B_N(u_N(t)))',u_N'(t))  $$
$$+ ((N_K^*(u_N(t)))',u_N'(t)) = (f'(t),u_N'(t)).$$
Since, $A$ is a linear operator $(Au_N(t))' = A(u_N'(t))$. So we get,
$$\frac{1}{2} \frac{d}{dt}|u_N'(t)|^2 + \nu \|u_N'(t)\|^2 + ((B_N(u_N(t)))',u_N'(t)) = (f'(t),u_N'(t))$$
$$ - ((N_K^*(u_N(t)))',u_N'(t)).$$
Now we will estimate $((B_N(u_N(t)))',u_N'(t))$ and $((N_K^*(u_N(t)))',u_N'(t))$. 
As $h \rightarrow 0$ we get,
$$((N_K^*(u_N(t)))',u_N'(t)) = \left( \frac{1}{h}(N_K^*(u_N(t+h)) - N_K^*(u_N(t))), \frac{1}{h} (u_N(t+h)-u_N(t)) \right)$$
which can be written as
\begin{eqnarray*}
((N_K^*(u_N(t)))',u_N'(t)) = \frac{1}{h^2} \left( N_K^*(u_N(t+h)) , u_N(t+h)-u_N(t) \right) \\
+ \frac{1}{h^2} \left( N_K^*(u_N(t)), u_N(t) - u_N(t+h) \right)
\end{eqnarray*}
Since, both the term in the right hand side are non-negative  we get, \\ 
$((N_K^*(u_N(t)))',u_N'(t)) \geq 0$.
\begin{eqnarray}\label{28}
\frac{1}{2} \frac{d}{dt}|u_N'(t)|^2 + \nu \|u_N'(t)\|^2 + ((B_N(u_N(t)))',u_N'(t)) \leq (f'(t),u_N'(t)). 
\end{eqnarray}
For $\|u_N\| \leq N$ we set,
\begin{align*}
((B_N(u_N(t)))',u_N'(t)) &= \frac{1}{h^2}(B_N(u_N(t+h))- B_N(t), u_N(t+h)-u_N(t)) \\
&= \frac{1}{h^2} [b(u_N(t+h),u_N(t+h)-u_N(t),u_N(t+h)-u_N(t) \\ 
&- b(u_N(t+h)-u_N(t),u_N(t),u_N(t+h)-u_N(t)))]. 
\end{align*}
As the limit $h \rightarrow 0$ we get,
\begin{eqnarray*}
((B_N(u_N(t)))',u_N'(t)) = b(u_N(t),u_N'(t),u_N'(t)) + b(u_N'(t),u_N(t),u_N'(t)).
\end{eqnarray*}
Using \ref{prop b} we get,
\begin{align*}
|((B_N(u_N(t)))',u_N'(t))| &\leq  |b(u_N'(t),u_N(t),u_N'(t)| \nonumber \\
&\leq C|u_N||u_N'|\|u_N'\|  \nonumber \\
&\leq C_\nu|u_N|^2|u_N'|^2 + \frac{\nu}{2}\|u_N'\|^2 . 
\end{align*}
For $\|u_N\| > N$ we get,
\begin{align*}
&((B_N(u_N(t)))',u_N'(t)) = \frac{1}{h^2}(B_N(u_N(t+h))- B_N(u_N(t)), u_N(t+h)-u_N(t)) \\
&= \frac{1}{h^2}\left(\frac{N^2}{\|u(t+h)\|^2} B(u_N(t+h))- \frac{N^2}{\|u(t)\|^2} B(u_N(t)), u_N(t+h)-u_N(t)\right) \\
&= \frac{1}{h^2}\left(\frac{N^2}{\|u(t+h)\|^2} B(u_N(t+h))- \frac{N^2}{\|u(t)\|^2} B(u_N(t+h)), u_N(t+h)-u_N(t)\right) \\
&+ \frac{1}{h^2}\left(\frac{N^2}{\|u(t)\|^2} B(u_N(t+h))- \frac{N^2}{\|u(t)\|^2} B(u_N(t)), u_N(t+h)-u_N(t)\right) \\
&= \left(B(u_N(t+h)) \frac{N^2}{h} \left( \frac{1}{\|u(t+h)\|^2} - \frac{1}{\|u(t)\|^2} \right) , \frac{u_N(t+h)-u_N(t)}{h} \right) \\
&+ \left(\frac{N^2}{\|u(t)\|^2} \frac{B(u_N(t+h))- B(u_N(t))}{h}, \frac{u_N(t+h)-u_N(t)}{h}\right). \\
\end{align*}
Now, as $h \rightarrow 0$ we get,
\begin{align*}
((B_N(u_N(t)))',u_N'(t)) &= \left(B(u_N(t)) \frac{2N^2\|u_N'(t)\|}{\|u_N(t)\|^3} , u_N'(t) \right) \\
&+ \frac{N^2}{\|u(t)\|^2} \left( B(u_N'(t)),u_N(t) \right) \\
= \frac{2N^2\|u_N'(t)\|}{\|u_N(t)\|^3} &b(u_N(t),u_N(t),u_N'(t)) + \frac{N^2}{\|u(t)\|^2} b(u_N'(t),u_N'(t),u_N(t)). 
\end{align*}
Since, $\|u_N(t)\| \neq 0$ we get,
\begin{align*}
|((B_N(u_N(t)))',u_N'(t))| \leq \frac{2N^2\|u_N'(t)\|}{\|u_N(t)\|^3}&|b(u_N(t),u_N(t),u_N'(t))| \\
&+ \frac{N^2}{\|u(t)\|^2}|b(u_N'(t),u_N(t),u_N'(t))| \\
\leq \frac{2\|u_N'(t)\|}{\|u_N(t)\|}|b(u_N(t),u_N'(t),u_N(t))| &+ |b(u_N'(t),u_N(t),u_N'(t))| \\
\leq C\frac{2\|u_N'(t)\|}{\|u_N(t)\|}|u_N(t)||u_N'(t)|\|u_N(t)\| &+ C|u_N'(t)||u_N(t)|\|u_N'(t)\| \\
\leq 2C\|u_N'(t)\||u_N(t)||u_N'(t)| &+ C|u_N'(t)||u_N(t)|\|u_N'(t)\| . 
\end{align*}
Using Young's inequality we get,
\begin{eqnarray*}
|((B_N(u_N(t)))',u_N'(t))| &\leq \frac{\nu}{2} \|u_N'(t)\|^2 + C_\nu|u_N(t)|^2|u_N'(t)|^2   
\end{eqnarray*}
So, from \eqref{28}, using above estimates we get,
\begin{align*}
\frac{d}{dt}|u_N'(t)|^2 + \frac{\nu}{2} \|u_N'(t)\|^2  \leq C_\nu|u_N(t)|^2|u_N'(t)|^2 + \frac{1}{2}|f'|^2 + \frac{1}{2}|u_N'(t)|^2 \\
\frac{d}{dt}|u_N'(t)|^2 + \frac{\nu}{2} \|u_N'(t)\|^2  \leq \left(\frac{1}{2}+C_\nu|u_N(t)|^2\right)|u_N'(t)|^2 + \frac{1}{2}|f'|^2.
\end{align*}
Therefore using Gronwall's inequality we get,
$$|u_N'(t)|^2  \leq \exp \bigg\{\int_0^t \left(\frac{1}{2}+C_\nu|u_N(s)|^2\right)ds\bigg\} \left(\frac{1}{2}\int_0^t |f'(s)|^2\right). $$
Then integrating  and using $u_N\in  W^{1,\infty}([0,T],H)$, $f \in W^{1,2}([0,T],H)$ and the above estimate we get,
$$|u_N'|^2 + \frac{\nu}{2} \int_0^t \|u_N'(s)\|^2 ds \leq C . $$
So the following inequality is true $\forall t\in [0,T]$ and $N=1,2,\cdot \cdot$, 
\begin{eqnarray}\label{e41}
|u_N'(t)|^2 + \int_0^T \|u_N'(s)\|^2 ds \leq C .
\end{eqnarray}
Now, taking inner product of \eqref{22} with $u_N$ we get,
$$\left( \frac{du_N}{dt}, u_N\right) + \nu \left(Au_N,u_N\right) + \left(B_N(u_N),u_N\right) + \left(N_K^*(u_N),u_N\right) = \left(f,u_N \right) .$$
which implies,
$$\nu \|u_N\|^2 \leq -\left( \frac{du_N}{dt}, u_N\right) + \left(f,u_N \right) .$$
Using similar estimates used to prove \eqref{e38} we can get,
$$\nu \|u_N\|^2  \leq  \frac{1}{2}|u_N'|^2 + |u_N|^2 + \frac{1}{2}|f|^2 .$$
Now using \eqref{e41} and $u_N\in W^{1,\infty}([0,T],H)$ and $f \in W^{1,2}([0,T],H)$ we get,
$$ \|u_N\| \leq  C.$$
So, for $N$ large enough, (such that $N > C$), $\|{u_N(t)}\| \leq C$ for all $t\in [0,T]$. Hence, we will get $B_N(u) = B(u) \ \forall \ t \in [0,T]$ and $u_N = u$ is the solution to \eqref{e17}. This completes the proof.
\end{proof}

 The next theorem is most general one where we wish to find a  control that  has support only in the small subset of the state space. In the case of sabra shell model it would be most interesting if we put the condition that the control be of finite dimension. For this model,  $H$ being the $l^2$ space, this condition would mean that we are projecting $u$ on a finite subset of $\mathbb{N}$ which in turn would mean that we are considering only finitely many modes of $u$. In the next theorem, we show that we can find controller such that corresponding solution to the closed loop system remains close to $K$, however it may not be in the $K$ for all $t$. The controls are written in terms of the projection operator on the set $K$.

\begin{thm}\label{thm 3}
Let, $K$ be a closed convex subset of $H$ and $0 \in K$. Let $P_K:H \rightarrow K$ be the projection operator on $K$ and $m$ is the characteristic function of a finite set $\omega \subset \mathbb{N}$ such that $P_K(mu)=mP_K(u) \ \ \forall u\in H$. Also assume $u^0 \in D(A)$ such that $mu^0 \in K$ and $f \in W^{1,2}([0,T],H)$. Then, for each $\lambda > 0$ there exists a feedback controller $g_\lambda = -\frac{1}{\lambda} \left( mu_\lambda - mP_K(u_\lambda) \right)$ such that the solution $u_\lambda$ of the closed loop system 
\begin{equation}\label{e25}
\frac{du(t)}{dt} + \nu Au(t) + B_N(u(t)) = f(t) + g(t), \ \ u(0)=u^0 
\end{equation}
satisfies, $u_\lambda \in W^{1,\infty}([0,T],H) \cap L^\infty([0,T],D(A))$. Moreover there exist a $C>0$ such that
\begin{eqnarray}\label{e26}
\frac{1}{\lambda} \int_0^T d_K^2(mu_\lambda(t))dt \leq C \ \ \forall \lambda > 0.
\end{eqnarray}
\end{thm}

\begin{proof}
Let the operator $\chi _N :D(\chi_N) \rightarrow H$ be defined by,
\begin{equation*}
\chi_N u = \nu Au + B_N(u) -\frac{1}{\lambda}(mu - P_K(mu)) . 
\end{equation*}
Clearly, $D(\chi_N)=D(A)$. \\
Define the operator $Fu = -\frac{1}{\lambda}(mu - P_K(mu))$. We can show that $F$ is non-expansive on $H$ i.e. 
\begin{eqnarray*}
|Fu - Fv| \leq |u-v|.
\end{eqnarray*}
Then by \cite{Barbu ns} we get, $\chi_N  + \alpha_N I$ where $\alpha_N >0$ is m-accretive on $H$. The proof will be similar as in \ref{thm 1}. Therefore we get, $u \rightarrow \chi_N u$ is a quasi m-accretive by definition \ref{QM}. So by \cite{Barbu ns} the cauchy problem, 
\begin{equation*}
\frac{du(t)}{dt} + \chi_Nu(t) = f(t), \ \ u(0)=u^0
\end{equation*}
has a unique solution $u_{\lambda}^N \in W^{1,\infty}([0,T],H) \cap L^\infty([0,T],D(A)) \cap C([0,T],V)$. \\
Taking inner product of 
\begin{equation}\label{27}
\frac{du_{\lambda}^N}{dt} + \chi_N u_{\lambda}^N = f, \ \ u(0)=u^0
\end{equation}
with $u^N_{\lambda}$ we get,
$$(\frac{du_{\lambda}^N}{dt},u^{\lambda}_N)  + (\chi_Nu_{\lambda}^N,u_{\lambda}^N) = (f,u_{\lambda}^N).$$ 
$$\mbox{  i.e.  }\frac{d}{dt}|u^N_{\lambda}|^2 + \nu (Au^N_{\lambda},u^N_{\lambda}) + \left(B_N(u^N_\lambda),u^N_{\lambda}\right) + \left(Fu^N_{\lambda},u^N_{\lambda}\right) = (f,u^N_{\lambda}) .$$
Since $\left(Fu^N_{\lambda},u^N_{\lambda}\right) \geq 0$ we get, 
\begin{eqnarray}\label{e42}
&\frac{d}{dt}|u^N_{\lambda}|^2 + \nu \|u^N_{\lambda}\|^2  \leq \frac{1}{2}|f|^2 + \frac{1}{2}|u^N_{\lambda}|^2   
\end{eqnarray}
Similarly taking inner product of \eqref{27} with $Au^N_{\lambda}$ and using 5. of Theorem \ref{prop b} we get,
\begin{eqnarray}\label{e43}
(\frac{du_{\lambda}^N}{dt},Au_{\lambda}^N) + \nu (Au^N_{\lambda},Au^N_{\lambda}) + \left(B_N(u^N_\lambda),Au^N_{\lambda}\right) + \left(Fu^N_{\lambda},Au^N_{\lambda}\right) = (f,Au^N_{\lambda}) \nonumber \\
\frac{d}{dt}\|u_{\lambda}^N\|^2 + \nu |Au_{\lambda}^N|^2 \leq C_\nu{|u_{\lambda}^N|}^2{\|u_{\lambda}^N\|}^2 + \frac{\nu}{6}|Au_{\lambda}^N|^2 + \frac{3}{2\nu}|Fu^N_{\lambda}|^2 + \frac{\nu}{6}|Au^N_{\lambda}|^2  \nonumber \\
+ \frac{3}{2\nu}|f|^2 + \frac{\nu}{6}|Au^N_{\lambda}|^2 \nonumber \\
\frac{d}{dt}\|u^N_{\lambda}\|^2 + \frac{\nu}{2} |Au_{\lambda}^N|^2 \leq C_\nu{|u_{\lambda}^N|}^2{\|u_{\lambda}^N\|}^2 + \frac{3}{2\nu}|Fu^N_{\lambda}|^2 + \frac{3}{2\nu}|f|^2.
\end{eqnarray}
Using Grownwall's lemma and using $f\in W^{1,2}([0,T],H)$, then integrating  \eqref{e42} gives,
$$|u^N_{\lambda}|^2 + \frac{\nu}{2} \int_0^t \|u^N_{\lambda}\|^2 \leq C \ \ \forall t \in [0,T].$$ 
Similarly integrating \eqref{e43} we get,
\begin{align*}
\|u^N_{\lambda}(t)\|^2 &\leq \|u^0\|^2 +C_\nu\int_0^t {|u_{\lambda}^N|}^2{\|u_{\lambda}^N\|}^2  + \int_0^t |Fu^N_{\lambda}(s)|^2  \\
&+  \frac{3}{2\nu}\int_0^t |f(s)|^2 
\end{align*}
Since $f\in W^{1,2}([0,T],H)$, $u^N_{\lambda} \in L^{\infty}([0,T],D(A))$ and  $Fu^N_\lambda$'s are bounded operators, it follows
$$\|u^N_{\lambda}(t)\|^2 \leq C_\nu\int_0^t {|u_{\lambda}^N|}^2{\|u_{\lambda}^N\|}^2 + \frac{C}{\lambda^2} + C ,$$
for all $ N \in \mathbb{N}$, $ \lambda > 0$ and $t \in [0,T]$. \\
Using Grownwall's inequality we get,
$$\|u^N_{\lambda}(t)\|^2 \leq C\left(1+ \frac{1}{\lambda^2}\right) \exp \{\int_0^t {|u_{\lambda}^N|}^2 \} $$
Since, $u^N_{\lambda} \in L^{\infty}([0,T],H)$, 
$$\|u^N_{\lambda}(t)\|^2 \leq C\left(1+ \frac{1}{\lambda} \right)^2 .$$
Therefore, for all $t \in [0,T]$ and $\lambda>0$ we get,
$$\|u^N_{\lambda}(t)\| \leq C\left(1+ \frac{1}{\lambda} \right) $$
So, for $N > C\left(1+ \frac{1}{\lambda} \right)$ we get, $u^N_\lambda = u_\lambda$, which is the solution of 
\begin{equation*}
\frac{du_\lambda(t)}{dt} + \nu Au_\lambda(t) + B_N(u_\lambda(t)) + F(u_\lambda(t))= f(t), \ \ u(0)=u^0, \forall \ t\in [0,T].
\end{equation*}
This proves \eqref{e25}. We are left to show \eqref{e26}. For we show that,
$$\left( mu_\lambda - mP_K(u_\lambda) , mu_\lambda - mP_K(u_\lambda) \right) \geq d_K^2(mu_\lambda),$$ 
which is same as to show that
$$\left( mu_\lambda , mu_\lambda - mP_K(u_\lambda) \right) - \left(mP_K(u_\lambda), mu_\lambda - mP_K(u_\lambda) \right) \geq d_K^2(mu_\lambda).$$ 
Clearly,
$$\left( mu_\lambda , mu_\lambda - mP_K(u_\lambda) \right) \geq d_K^2(mu_\lambda) $$
Since, $Fu_\lambda = \frac{1}{\lambda} \left( mu_\lambda - mP_K(u_\lambda) \right),$
$$\left( mu_\lambda , \lambda Fu_\lambda \right) \geq d_K^2(mu_\lambda)  $$
i.e
$$\left( mu_\lambda , Fu_\lambda \right) \geq \frac{1}{\lambda} d_K^2(mu_\lambda).$$ 
Using Cauchy-Schwartz and Young's inequality we get,
$$\frac{1}{\lambda} d_K^2(mu_\lambda) \leq \frac{|Fu_\lambda|^2}{2} + \frac{|u_\lambda|^2}{2}.$$
Integrating over $[0,T]$ gives,
$$\frac{1}{\lambda} \int_0^T d_K^2(mu_\lambda) \leq \int_0^T \frac{|Fu_\lambda|^2}{2} + \int_0^T \frac{|u_\lambda|^2}{2} .$$
Since, $u_{\lambda} \in L^{\infty}([0,T],H)$, for all $ \lambda > 0$ there exists $C>0$ such that 
$$\frac{1}{\lambda} \int_0^T d_K^2(mu_\lambda) \leq C.$$ 
This proves the theorem.
\end{proof}

Now we illustrate with two examples how we can find flow preserving controllers using theorems proved above. The quantities like enstrophy and helicity  are conserved for fluid flow equations. In \cite{Peter}, it has been shown that shell model of turbulence also preserves these quantities.

The first example is concerned with  enstrophy of a system. We choose convex set $K$ to be an enstrophy ball,  and show that we can find a control which would preserve the flow if the initial value is  assumed to be in $K$. The first theorem is applicable in this case and we explicitly define the feedback control which would ensure that solution does not leave $K$. In the second example, we consider $K$ to be a subset of $H$ for which helicity is bounded by some bound say $\rho$.  Here, we can show that the first theorem can not be applied and we can use the third theorem to ensure that the helicity of the solution remains near $K$ for all time $t$ almost everywhere.

\begin{ex}
\textbf{Enstrophy Invariance:} 
Consider the set $K= \{u\in V;\|u\|=|A^{1/2}u| \leq \rho \}$. We want to find $N_K( u )$ such that if the control belongs to $N_K( u )$, the solution will be in $K$.
\end{ex} 
Clearly, $K$ is closed, convex subset of $H$. We show that using  Theorem \ref{thm 1} we can find a feedback controller so that the solution of \eqref{e4} does not leave $K$.
First we show that $(I+\lambda A)^{-1} K \subset K$. \\
For any $f \in K$ we have to show there exist $y\in K$ such that 
\begin{equation}\label{e28}
(I+\lambda A)y = f .
\end{equation}
Taking inner product of \eqref{e28} with $Ay$ we get, 
\begin{equation*}
(y,Ay) + \lambda(Ay,Ay) = (f,Ay).  
\end{equation*}
Using Young's inequality we get,
\begin{equation*}
|A^{1/2}y|^2 + \lambda |Ay|^2 \leq \frac{1}{2}|A^{1/2}f|^2 + \frac{1}{2} |A^{1/2}y|^2.
\end{equation*}
Therefore we get,
\begin{equation*}
|A^{1/2}y|^2 \leq |A^{1/2}f|^2 \leq \rho.
\end{equation*}
Thus the solution $y$ to \eqref{e28} belongs to $K$. Therefore $(I+\lambda A)^{-1} K \subset K$. We have
\[N_K( u ) := \left\{ w\in H ;
  \begin{cases}
    0      & \quad \mbox{if } |A^{1/2}u| < \rho\\
   \cup  _{ \lambda > 0} \lambda Au  & \quad \mbox{if } |A^{1/2}u| = \rho\\
  \end{cases} \right\}
\] 
Using the theorem \ref{thm 1} we get the feedback controller 
\[g(t) := 
  \begin{cases}
    0      & \quad \mbox{if } |A^{1/2}u| < \rho\\
    Z(u) & \quad \mbox{if } |A^{1/2}u| = \rho\\
  \end{cases} 
\]
where $Z(u) = -\frac{Au}{|Au|^2} (f(t) -\nu Au(t) - B(u(t)), Au(t)) $ such that the corresponding solution $u(t) \in K \ \forall t\in [0,T]$. \\

\begin{ex}
\textbf{Helicity Invariance: }Let's consider $K=\{y \in D(A^{1/4}) \ ; \sum k_n|y_n|^2 = |A^{1/4}y|^2 \leq \rho^2 \}$.
\end{ex}
Clearly, $K$ is closed convex subset of $H$. We can show that $K$ is not invariant under $(I+\lambda A)^{-1}$. So we can not use the Theorem \ref{thm 1}. We are going to apply Theorem \ref{thm 3}. \\
Let $L: V_{1/2} \rightarrow H$ be the operator defined by $ L(y)=A^{1/4}y$. Then the functional $N(y) = |L(y)|^2$ is continuous on $V_{1/2}.$ \\
So using Theorem 5.3 we get a sequence of feedback controller $g_\lambda \in L^2([0,T],H)$ such that $g_{\lambda} = \frac{1}{\lambda}(u_{\lambda} - P_K (u_\lambda))$ and $$ \lim_{\lambda \rightarrow 0} \int^T_0 d_K^2(u_{\lambda}(t)) dt = 0.$$ 
Now to write $g_{\lambda}$, we need to find projection of $u_{\lambda}$ on $K$. Let us denote $P_K(u)=z$ for any $u\in V_{1/2}$. We will show that $z$ solves
\begin{eqnarray}\label{e44}
z+ 4\lambda L(z)A^{1/2}z = u , \ \ z(0)=z^0
\end{eqnarray}
where $z^0$ is the projection of $u^0$ on $K$. To prove that $z$ solves \eqref{e44} we use Lagrange multiplier method. Define the minimization problem,
\begin{equation*}
\begin{aligned}
& \underset{x \in K}{\text{minimize}}
& & d(x,u) \\
& \text{subject to}
& & N(x) \leq \rho^2.
\end{aligned}
\end{equation*}
Let the solution of the minimization problem be $z$. Then $z$ will satisfy the following Lagrange multiplier equation.
\begin{equation}\label{29}
z-u + \lambda N'(z) = 0.
\end{equation}
Now to find $N'(z)$, it is enough to find $L'(z)$ as $N'(z) = 2L(z)L'(z).$
\begin{align*}
L'(z)h &= lim_{\lambda \rightarrow 0} \frac{L(z+ \lambda h) - L(z)}{\lambda} \\
&=  lim_{\lambda \rightarrow 0} \frac{\sum k_n (z_n + \lambda h_n)^2 - \sum k_n z_n^2}{\lambda} \\
&= lim_{\lambda \rightarrow 0} \frac{\sum k_n 2 \lambda z_n h_n + \sum k_n {\lambda}^2 h_n^2}{\lambda} \\
&= 2 \sum k_n z_n h_n \\
&= 2(A^{1/4}z,A^{1/4}h)
\end{align*}
Since, $D(A^{1/2})$ is dense in $D(A^{1/4})$ we can extend the definition of $A^{1/2}:V_{1/2} \rightarrow V_{-1/2}$ in such a way that
$$ \langle A^{1/2}u,v \rangle = (A^{1/4}u,A^{1/4}v).$$
So we get, $$L'(z) = A^{1/2}z.$$
Therefore, by \eqref{29}  $z$ solves \eqref{e44} i.e. $$z+ 4\lambda L(z)A^{1/2}z = u .$$

\section{Conclusion}
In this work we have studied  two optimal control problems for equations of sabra shell model of turbulence with the control acting as a forcing term. We have studied two cost functionals, one aims to reduce turbulence in the flow and the other one  is to find an optimal control which can take the flow to the desired state.  In both cases, with the help of the adjoint equation we have shown that if the optimal pair which minimizes the cost functional exists then the optimal control can be characterized by  using the solution of appropriate the adjoint equation. However, converse of our theorem does not hold true. That is the  control designed by above method need not give the optimal solution for the minimization of cost functional. This is expected because the adjoint equation is written for the linearized equations of sabra shell model. Thus for nonlinear equation the control which is designed via solution of the linearized adjoint equation would not give the optimal solution. 

In the second part of our work we have looked at another control problem of  finding feedback controller which would  preserve certain quantities in the flow. This is mainly useful because the shell models have natural invariants like enstrophy and helicity. We have proved three  different theorems about constructing feedback controllers. The first theorem is proved under strong assumption that the constrained set $K$ is invariant under $(I+\lambda A)^{-1}$. In this case we get the  control which  will be in $H$-valued normal cone to $K$. In the second theorem the  assumption is relaxed, but as a result, we get the control in the $V$ valued normal cone to $K$ which is a  weaker space than before. In the third theorem, we  have proved that we can always find a sequence of controls such that the sequence of corresponding solutions will remain close to $K$. At the end, we have discussed two example where these theorems have been applied.

As noted earlier  control problems for shell models of turbulence are not well studied in literature in spite of potential applications. We further plan to study  certain controllability problems related to sabra shell model of turbulence. Internal stabilization and H infinity control problem for the shell model also seem promising avenues to explore.

\end{document}